\documentclass[11pt]{amsart}

\usepackage{amsmath}
\usepackage{amssymb}
\usepackage{euscript}
\usepackage{graphicx}

\def\Cale#1{\EuScript #1}
\def\PP{\Cale P}

\def\D{\mathrm{d}}

\def\E{\mathrm{e}}

\newcommand{\DS}{\displaystyle}

\newtheorem{theorem}{Theorem}[section]

\newtheorem{lemma}[theorem]{Lemma}

\theoremstyle{definition}

\newtheorem{remark}[theorem]{Remark}

\numberwithin{equation}{section}

\title[On the Markov extremal problem in the  $L^2$-norm]{On the Markov extremal problem in the  \boldmath$L^2$-norm with the classical weight functions}

\author[G. V. Milovanovi\'c]{Gradimir V. Milovanovi\'c}

\address[G. V. Milovanovi\'c]{Serbian Academy of Sciences and Arts, 11000 Beograd, Serbia \&
Faculty of Science and Mathematics, University of Ni\v s, 18000 Ni\v s, Serbia}
\email{{\tt gvm@mi.sanu.ac.rs}}

\keywords{Markov-Bernstein's extremal problems; Best constant; Orthogonal polynomials; Inner product; Eigenvalue problems.}

\subjclass[2020]{41A44, 33C45, 41A17, 65F15, 42C05}

\begin{document}

\begin{abstract}
This paper is devoted to Markov's extremal problems of the form $M_{n,k}=\sup_{p\in\PP_n\setminus\{0\}}{{\|p^{(k)}\|}_X}/{{\|p\|}_X}$ $(1\le k\le n)$, where $\PP_n$ is the set of all algebraic polynomials of degree at most $n$  and $X$ is a normed space, starting with original Markov's result in uniform norm on $X=C[-1,1]$  from the end of the 19th century. The central part is devoted to extremal problems on the space $X=L^2[(a,b);w]$ for the classical weights $w$ on $(-1,1)$, $(0,+\infty)$  and 
$(-\infty,+\infty)$. Beside a short  account on basic properties of the (classical) orthogonal polynomials on the real line, the explicit formulas for expressing $k$-th derivative of the classical orthonormal polynomials in terms of the same polynomials are presented, which are important in our study of this kind of extremal problems, using methods of linear algebra. Several results for all cases of the classical weights, including   algorithms for numerical computation of the best constants $M_{n,k}$, as well as their lower and upper bounds, asymptotic behaviour, etc.,
are also given. Finally, some  results on Markov's extremal problems on certain restricted classes of polynomials are also mentioned. 
\end{abstract}

\maketitle


\section{Introduction and Preliminaries}

Inequalities for polynomials and their derivatives, as well as the corresponding extremal problems, are very important in many areas in mathematics, but also in other computational and applied sciences. In  particular they play a fundamental rule in  {\it Approximation Theory}, e.g., inequalities of Markov and Bernstein-type  are fundamental for the proofs of many inverse theorems in the so-called
{\it Polynomial Approximation Theory}.  These inequalities and extremal problems can be considered in different, usually normed spaces. Several monographs  have been published in this area (cf.\ \cite{mil1994,bor1995Book,rah2002}), as well as many papers\footnote{This paper is dedicated to the mathematical contribution of  Professor Ravi Agarwal.}. 

 In this paper by $\PP_n$ we denote the set of all algebraic polynomials of degree at most $n$, and by 
 $\Pi_n$  the set of all  polynomials of exact degree $n$, so that
$\PP_n=\bigcup_{k=0}^n \Pi_k$, and 
 $\widehat\Pi_n$ will be the 
set of all monic polynomials of degree $n$,  i.e., 
\[\widehat\Pi_n=
\Bigl\{t^n+q(t)\,\mid\, q(t)\in \PP_{n-1}\Bigr\}\subset \Pi_n.\]

The first result in this area was 
connected with some investigations of the 
well-known Russian  chemist  Dmitri Mendeleev (1834--1907). 
In mathematical terms, Mendeleev's problem \cite{men1887} was as follows:
{\it If $t\mapsto P(t)$ is an arbitrary quadratic polynomial defined on an interval $[a,b]$, with 
$$\max\limits_{t\in [a,b]}P(t)-\min\limits_{t\in [a,b]}P(t)=L,$$ 
how large can $P'(t)$ be on $[a,b]$}? It can be reduced to a simpler problem by changing the horizontal scale and shifting the coordinate axis until we have $|P(t)|\le 1$, so that Mendeleev's problem becomes the following:
{\it If $t\mapsto P(t)$ is an arbitrary quadratic polynomial and
$|P(t)|\le 1$ on $[-1,1]$, how large can $|P'(t)|$ be on $[-1,1]$}?
 
Mendeleev found that $|P'(t)|\le 4$ on $[-1,1]$. This result is the best possible because for $P(t)=1-2t^2$ we have $P(t)\le 1$ and $P'(\pm 1)=4$. The corresponding problem for polynomials from $\PP_n$ was solved by a very famous Russian academician Andrei Andreyevich Markov (1856--1922) at the end of the 19th century. Markov's younger half-brother Vladimir Andreevich Markov  (1871--1897), although he died young, gained also an international reputation because he later solved the  problem for $k$-th derivative 
  $|P^{(k)}(t)|$, $k>1$. Both were students of the famous Pafnuty Lvovich Chebyshev  (1821--1894) at  St.\ Petersburg State University. Another member of  the Russian mathematical school, a student of the French Sorbonne  and of Jewish origin, is  Sergei Natanovich Bernstein  (1880--1968), whose results have left a deep mark on the development of this field. 
  
We mention here their basic results:  

\begin{theorem}[A.A. Markov \cite{mar1889} (1889)]\label{AAMarkov1889}
If $p\in\PP_n$ such that $|p(x)| \le 1$ on $[-1,1]$, then
\begin{equation}\label{markov-eqn}
|p'(x)| \le n^2 \ \ \mbox{ for }\ -1 \le x \le 1.
\end{equation}
This inequality is best possible and the equality is attained at only $x = \pm 1$, and only when $p(x) = \gamma T_n(x)$, where $\gamma$ is a complex number such that $|\gamma|=1$, and $T_n(x)$ is well-known Chebyshev polynomial of the first kind, defined by 
\[
T_n(x) = \cos(n \cos^{-1} x) = 2^{n-1} \prod\limits_{\nu = 1}^n \left\{x -\cos\frac{(2\nu-1)\pi}{2n}\right\}.	 
\] 
\end{theorem}

Introducing the {\it  uniform norm} for polynomials on $[-1,1]$ as
\[
\|p\|_\infty=\|p\|_{[-1,1]}=\max_{x\in[-1,1]}|p(x)|,\]	
then  Markov's result  can be expressed in the form 
\begin{equation}\label{MarkovEP}
\sup_{p\in\PP_n\setminus\{0\}}\frac{\|p'\|_\infty}{\|p\|_\infty}=T_n(1)=n^2.
\end{equation}
It is known as {\it Markov's extremal problem} in the uniform norm.

\begin{theorem}[V.A. Markov \cite{Vmar1892,mar1916}\ (1892;\,1916)]\label{VAMarkov1892}
For $1\le k\le n$, we have 
\begin{equation}\label{MarkovEPk}
\sup_{p\in\PP_n\setminus\{0\}}\frac{\|p^{(k)}\|_\infty}{\|p\|_\infty}= T^{(k)}_n(1),
\end{equation}
where
\[T^{(k)}_n(1)=\frac{n^2(n^2-1^2)(n^2-2^2)\cdots
(n^2-(k-1)^2)}{(2k-1)!!}.\]
Extremal polynomial is  $p(x)=\gamma T_n(x)$, with $|\gamma |=1$.
\end{theorem}

\begin{theorem}[S.N. Bernstein \cite{ber1912} (1912)]\label{SNBernstein1912} 
If $p\in\PP_n$ and $|p(x)| \le 1$ for $-1 \le x \le 1$, then
\begin{equation}\label{bern-eqn5bern}
|p'(x)| \le \frac{n}{\sqrt{1 - x^2}}, \quad -1 < x < 1.
\end{equation}
The equality is attained at the Chebyshev points 
$x = x_\nu = \cos \frac{(2\nu - 1)\pi}{2n}$,  $1 \le \nu \le n$, 
if and only if $p(x) = \gamma T_n(x)$, where $|\gamma| = 1$,  
and $(\ref{bern-eqn5bern})$ is best possible.
\end{theorem}

Note that this version of Bernstein's inequality is a pointwise inequality, while Markov's inequality  
$|p'(x)| \le n^2$, $-1 < x < 1$, 
 is global. Combining these inequalities, we get
\[|p'(x)|\le \min\left\{n^2,\,\frac{n}{\sqrt{1-x^2}}\right\},
\qquad -1\le x\le 1.\]
The natural question is how large can $|p'(x)|$  be for a given $x\in[-1,1]$, when $p\in\PP_n$ and 
$|p(x)|\le 1$ on $[-1,1]$?  Such a function  $x\mapsto M_n(x)$  is evidently an even function on $[-1,1]$. Explicit expressions for $n=2$ and $n=3$ can be found in \cite[pp.~539]{mil1994}. The determination of $M_n(x)$ for $n\ge 4$  is very complicated and it can be given by a technique of Voronovskaja \cite{vor1970}.

Taking norms different from the {\it  uniform norm} we can consider 
Markov's extremal problem in other spaces, e.g.\ in $L^r$ $(r\ge 1)$, or even in spaces with {\it quasi-norms} $L^r$ $(0\le r< 1)$, etc. 
One can also consider the so-called {\it  mixed Markov type inequalities} on $[-1,1]$,
\[
{\|P^{(k)}\|}_r\le M_{r,q}(n,k) {\|P\|}_q,\quad 0\le k\le n;\ 0\le r,q\le+\infty,	
\]
when $P\in{\PP}_n$, where 
\begin{align*}
&{\|P\|}_r=\left(\frac{1}{2}\int_{-1}^1 |P(t)|^r{\D}t\right)^{1/r}\quad (0<r<+\infty),\\
&{\|P\|}_0 =\lim_{r\to0+}{\|P\|}_r
=\exp\left(\frac{1}{2}\int_{-1}^1\log|P(t)|\,{\D}t\right),\\
&{\|P\|}_\infty=\max_{-1\le t\le 1}|P(t)|.	
\end{align*}
Evidently, ${\|P\|}_r$ for $0<r<1$ is a {\it quasi-norm} of $P$. For some extremal problems of this type see \cite{Kon,boj1982-3,Glazyrina2005,Glazyrina2005Zam,Glazyrina2008Zam,Simonov2012}. 

For $r\ne q$ and $k=0$ the inequality  is know as the  {\it  Nikol'ski\u\i\ inequality} (see \cite[pp.~495--507]{mil1994}).
For $k = n$, the previous problem  
reduces to finding a polynomial that deviates least from zero in the $L^q$-metric 
with a fixed leading coefficient. For example, when $q=+\infty$, $q=2$, and $q=1$, the solutions (extremal polynomials) are known: the Chebyshev polynomial of the first kind $T_n(t)$, the Legendre polynomial $P_n(t)$, and the Chebyshev polynomial of the second kind $U_n(t)$, respectively.  

However, the set $\PP_n$ can be restricted to some of subsets $W_n\subset \PP_n$, and then we can consider the corresponding Markov extremal problem on such restricted sets again in different norms (cf.~\cite[Chapters 5 \& 6]{mil1994}).

In this paper we consider Markov's extremal problems of the form
\begin{equation}\label{MEProb}
M_{n,k}=\sup_{p\in\PP_n\setminus\{0\}}\frac{{\|p^{(k)}\|}_X}{{\|p\|}_X}
\quad(1\le k\le n),	
\end{equation}
on the inner product functional space $X=L^2[(a,b);w]$, with the inner product defined by 
\begin{equation}\label{innPrCOP}
(p,q)_w=\int_a^b p(t)\,\overline{q(t)}\,w(t)\,{\D}t,	
\end{equation}
where $t\mapsto w(t)$ is a non-negative  function  on  $(a,b)$, $-\infty\le a<b\le +\infty$,
for which all moments $\mu_k=\int_a^b t^k w(t)\,{\D}t$, $k=0,1,\ldots$, exist and $\mu_0>0$. Such a function is known as the {\it weight function} on  $(a,b)$.
The norm of an element $p\in X$ is given by
\begin{equation}\label{normWX}
\|p\|_X=\sqrt{(p,p)_w}=\left(\int_a^b|p(t)|^2w(t)\,{\D}t\right)^{1/2}.
\end{equation}
For $M_{n,k}$ in \eqref{MEProb} we use  terms the {\it best}, {\it exact} or {\it sharp constant}.
In particular, we treat the basic extremal problem for the first derivative, i.e., the determination of the  best constant $M_{n,1}\equiv M_n$.

The paper is organized as follows.  A short  account on 
basic properties of the orthogonal polynomials on the real line, and in particular for ones known as the ``classical orthogonal polynomials'', is given in Section \ref{SEC2}. Explicit formulas for expressing $k$-th derivative of the classical orthonormal polynomials in terms of the same polynomials are presented in Section \ref{SEC3}. Such formulas are important in our study of extremal problems \eqref{MEProb} on $X=L^2[(a,b);w]$ for the classical weight functions $w$ on $(-1,1)$, $(0,+\infty)$, and $(-\infty,+\infty)$ in Section \ref{SEC4}. 
Special  cases of  $L^2$ Markov's extremal problems for all classical weight functions are given in Sections \ref{SEC5}--\ref{SEC7}. Finally, in Section  \ref{SEC8} some  results on Markov's extremal problems on certain restricted classes of polynomials are mentioned.  

\section{Basic Properties of the Orthogonal and the Classical Orthogonal Polynomials}\label{SEC2}

The  orthogonal polynomials are basic tools in the investigation of the extremal problems \eqref{MEProb} on the space $X=L^2[(a,b);w]$, and therefore in the this section we give some basic properties of the orthogonal polynomials on the real line, including, in particular, an important class of the so-called {\it very classical orthogonal polynomials} (cf.\ \cite{Chi1978,Gau2004,gmgvm2008}).

The inner product (\ref{innPrCOP}) gives rise to a unique system of {\it orthonormal polynomials} $p_n(\,\cdot\,)=p_n(\,\cdot\,;w)$,  such that
$p_n(t)=\gamma_n t^n+ \mbox{terms of lower degree}$, with 
$\gamma_n>0$ for each $n\in\mathbb{N}$, 
and
\begin{equation}\label{orthnp}
(p_k,p_n)_w=\int_a^b p_k(t)\,\overline{p_n(t)}\,w(t)\,{\D}t=\delta_{kn}, \quad k,n\ge0.
\end{equation}
Also, we need here the {\it monic orthogonal polynomials},
in notation, 
\[\pi_n(t)=\pi_n(t;w)=\frac{p_n(t)}{\gamma_n}= t^n + \mbox{terms of lower degree}.\] 

Because of the property of the inner product $(tp,q)_w=(p,tq)_w$, orthogonal polynomials on the real line satisfy a {\it three-term recurrence relation} (cf. \cite[p.~99]{gmgvm2008}).

{\rm(a)} For orthonormal polynomials $p_n(t)$ we have
\begin{equation}\label{eq:ttrrORTNOR}
 tp_n(t)=b_{n+1}p_{n+1}(t)+a_np_n(t)+b_np_{n-1}(t),\quad n=0,1,\ldots, 
  \end{equation}
with $p_0(t)=\gamma_0=1/\sqrt{\mu_0}$ and $p_{-1}(t)=0$,  where the coefficients $a_n=a_n(w)$ and $b_n=b_n(w)$ are given by
\[a_n=(tp_n,p_n)_w\quad\mbox{and}\quad 
b_n=(p_n,tp_{n-1})_w=\frac{\gamma_{n-1}}{\gamma_n}>0;\]

{\rm(b)} For monic orthogonal polynomials $\pi_n(t)$ we have
\begin{equation}\label{eq:ttrrORTMon}
\pi_{n+1}(t)=(t-\alpha_n)\pi_n(t)-\beta_n \pi_{n-1}(t),\quad n=0,1,\ldots, 
  \end{equation}
with $\pi_0(t)=1$ and $\pi_{-1}(t)=0$, where the coefficients $\alpha_n=\alpha_n(w)$ and $\beta_n=\beta_n(w)$ are given by
\[\alpha_n=a_n=\frac{(t\pi_n,\pi_n)_w}{(\pi_n,\pi_n)_w}\quad(n\ge0),\quad  
\beta_n=b_n^2=\frac{(\pi_n,\pi_n)_w}{(\pi_{n-1},\pi_{n-1})_w}>0\quad(n\ge1).\]

These coefficients in the three-term recurrence relations (\ref{eq:ttrrORTNOR}) and (\ref{eq:ttrrORTMon})
depend only on the weight function $w$. The coefficients $\beta_k$, $k\ge 1$, in (\ref{eq:ttrrORTMon}) are positive, and  $\beta_0$ may be arbitrary, but  sometimes it is convenient to define it by $\beta_0=\mu_0=\int_a^b w(t)\,{\D}t$. Then, it is easy to see that 
\[\|\pi_n\|_X=\sqrt{(\pi_n,\pi_n)_w}=\sqrt{\beta_0\beta_1\cdots\beta_n},\]
where the norm is given by (\ref{normWX}).

An important result on zero distribution of orthogonal polynomials on the real line  is the following  (cf. \cite[p.~99]{gmgvm2008}):

All zeros of $\pi_n(t)$, $n\in\mathbb{N}$, are real and distinct and are located in the interior of $(a,b)$. Furthermore, the zeros of  $\pi_n(t)$ and $\pi_{n+1}(t)$ interlace, i.e.,
\[\tau_{n+1,\nu}<\tau_{n,\nu}<\tau_{n+1,\nu+1},\quad \nu=1,\ldots,n.	\]
Here,  $a<\tau_{n,1} < \tau_{n,2} < \cdots <\tau_{n,n}<b$ denote the zeros of $\pi_n(t)$ in increasing order.

Using procedures of numerical linear algebra, notably the QR or QL algorithm, it is easy to compute the zeros of the orthogonal polynomials $\pi_n(t)$ rapidly and efficiently as eigenvalues of the {\it Jacobi matrix} of order $n$ associated with the weight function $w$,
\begin{equation}\label{JacobiMAT}
 J_n(w)=\left[\begin{array}{ccccc}
 \alpha_0&\sqrt{\beta_1}&&&{\mathbf{O}}\\[1mm]
  \sqrt{\beta_1}&\alpha_1&\sqrt{\beta_2}\\
 &\sqrt{\beta_2}&\alpha_2&\ddots  \\
  &&\ddots&\ddots&\sqrt{\beta_{n-1}}\\[1mm]
{\mathbf{O}}&&&\sqrt{\beta_{n-1}}&\alpha_{n-1}
\end{array}\right].
 \end{equation}
Unfortunately, the recursion coefficients $\alpha_n$ and $\beta_n$ in (\ref{eq:ttrrORTMon})  are known explicitly only for some narrow classes of orthogonal polynomials. One of the most important 
classes for which these coefficients are known explicitly are surely the 
so--called {\it very classical} orthogonal polynomials, which appear frequently in  applied analysis and  computational sciences. Orthogonal polynomials for which the recursion coefficients are not known we call {\it strongly non--classical polynomials}. 

\subsection{Classical weight functions and the corresponding orthogonal polynomials}\label{orthCW}

In the sequel we consider only very classical orthogonal polynomials, omitting   the term ``very'' and call them simply the {\it classical orthogonal polynomials}. They are distinguished by several particular properties
(cf. \cite[pp.~121--146]{gmgvm2008}).
 Their weight function, the so-called  {\it classical weight function} on $(a,b)$, satisfies a first order differential equation of the form  
 \[\frac{\D}{{\D}t}(A(t)w(t))=B(t)w(t),\] 
where $B(t)$ is a first degree polynomial, and $A(t)$ is one of degrees not greater than two. For such classical weights we will write $w\in CW$, and for the classical orthogonal polynomials use a general notation $Q_n(t)$. We note
that for such weights $w\in CW$, we have $w\in C^1(a,b)$, as well as
\[\lim_{t\to a+}t^m A(t)w(t)=0\quad\mbox{and}\quad \lim_{t\to b-}t^m A(t)w(t)=0\qquad (m=0,1,\ldots).\]

Without loss of generality, the classical  polynomials  orthogonal with respect to the inner product (\ref{innPrCOP}), can be considered only on 
three different intervals: $(-1,1)$, $(0,+\infty)$, and $(-\infty,+\infty)$, because every interval $(a,b)$ can be transformed by a linear transformation to one of the previous intervals. These three cases are presented in Table~\ref{ClassOP}.
\begin{table}[h]
\caption{Classification of the classical orthogonal polynomials.}
\label{ClassOP} 
\begin{center}
{\setlength{\tabcolsep}{6pt}
\renewcommand{\arraystretch}{1.2} 
\begin{tabular}{|c|c|c|c|c|}\hline 
$(a,b)$ &  $w(t)$  &  $A(t)$  & $B(t)$  & $Q_n(t)$ \\ \hline  
$(-1,1)$  &  $(1-t)^\alpha(1+t)^\beta$  &  $1-t^2$ & 
$\beta-\alpha-(\alpha+\beta+2)t$   
& $P_n^{(\alpha,\beta)}(t)$ \\ 
$(0,+\infty)$ & $t^s{\E}^{-t}$ & $t$ & $s+1-t$  & $L_n^s(t)$ \\
   $(-\infty,+\infty)$ & ${\E}^{-t^2}$ & $1$ & $-2t$  & $H_n(t)$\\ \hline 
\end{tabular} }
\end{center} 
\end{table} 

The corresponding orthogonal polynomials are known as  the {\it Jacobi polynomials} $P_n^{(\alpha,\beta)}(t)$ \  $(\alpha,
\beta>-1)$, the {\it generalized Laguerre polynomials} $L_n^s(t)$ \ $(s>-1)$,
and finally as the {\it Hermite polynomials} $H_n(t)$. Because of the existence of the moments, the parameters $\alpha,\beta$, and $s$ should be greater than $-1$.  The classical orthogonal polynomial $t\mapsto Q_n(t)$ is a particular
solution of the following differential equation 
$L[y]\equiv A(t)y''+B(t)y' +\lambda_ny=0$,	
where 
\begin{equation}\label{eq:lamn}
\lambda_n=-n\Bigl(\frac12(n-1)A''(0)+B'(0)\Bigr). 	
\end{equation}
The corresponding values are $\lambda_n=n(n+\alpha+\beta+1)$ for the Jacobi polynomials, $\lambda_n=n$  for the generalized Laguerre polynomials, and $\lambda_n=2n$ for the Hermite polynomials.

The corresponding orthonormal classical polynomials will be denoted by small letters $q_n(t)$; in particular, by  $p_n^{(\alpha,\beta)}(t)$, $l_n^s(t)$, and $h_n(t)$, and the monic polynomials as $\widehat Q_n(t)$, i.e., by $\widehat P_n^{(\alpha,\beta)}(t)$, $\widehat L_n^s(t)$, and $\widehat H_n(t)$.

There are several characterizations of the classical orthogonal polynomials (cf.\ \cite{AlSChi1972}). One of them was given by Agarwal and Milovanovi\'c \cite{BC4,RJ110}:

\begin{theorem}\label{thm:AM} Let $w\in CW$ and  $X=L^2[(a,b);w]$.  Then for all $P(t)\in \Cale P_n$ the inequality
\begin{equation}\label{eq:AM}
(2\lambda_n+B'(0))\,\bigl\|\sqrt{A}P'\bigr\|_X^2 \le \|A P''\|_X^2 +
                   \lambda_n^2\,\| P\|_X^2
\end{equation}
holds, with equality if only if $P(t)=cQ_n(t)$, where $Q_n(t)$ is the classical orthogonal polynomial and $c$ is an arbitrary constant.
\end{theorem}

An important property of the classical orthogonal polynomials is the following 
result (cf. \cite[pp.~124--126]{gmgvm2008}):

\begin{theorem}\label{izvVR}
 The derivatives of the classical orthogonal polynomials $Q_n(t)$, $n\in\mathbb{N}$, with respect to the weight function $t\mapsto w(t)$  $(w\in CW)$,  also form a sequence of the classical orthogonal polynomials $Q'_n(t)$
 with respect to the weight function $t\mapsto w_1(t)=A(t)w(t)$ $(w_1\in CW)$.	
 \end{theorem}
 
 According to Theorem \ref{izvVR} and the uniqueness of orthogonal polynomials, the following formulas
\begin{align}
\frac{\D}{{\D}t}P_n^{(\alpha,\beta)}(t)&=\frac{1}2(n+\alpha+\beta+1)P_{n-1}^{(\alpha+1,\beta+1)}(t)	,\label{izvPJ}\\[1mm]
\frac{\D}{{\D}t}L_n^{s}(t)&=-L_{n-1}^{s+1}(t),\label{izvPL}\\[1mm]
\frac{\D}{{\D}t}H_n(t)&=2nH_{n-1}(t)\label{izvPH}
\end{align}
hold.

\section{Differentiation Formulas for the Classical Orthogonal Polynomials}\label{SEC3} 

In this section we give formulas for expressing $k$-th derivative of the orthonormal classical polynomials $q_n$
in terms of the same polynomials, i.e.,
\[\frac{\D^k}{{\D}t^k}q_n(t)=\sum_{\nu=0}^{n-k}c_{\nu,n}^{(k)}\,q_\nu(t),\quad k\le n.\]
For all classical polynomials we can get explicit formulas for the coefficients $c_{\nu,n}^{(k)}$. Such expressions we use in our study of extremal problems \eqref{MEProb}, when $w\in CW$. Namely, in our consideration of the $L^2$ Markov extremal problems for the classical weights  we usually reduce them to  eigenvalue problems on the finite-dimensional spaces generated by orthonormal polynomials. Therefore, we are interested in expressing the  derivatives of the basis polynomials as a linear combination of exactly the same polynomials.

Now, we separately consider three classical cases (see  
Table~\ref{ClassOP}).
\smallskip

{\bf Hermite polynomials.} In this case, using \eqref{izvPH} we get 
\[h'_n(t) =\sqrt{2n}\,h_{n-1}(t),\quad   h''_n(t)=2\sqrt{n(n-1)}\,h_{n-2}(t),\quad\mbox{etc.}\] 

Thus,
\begin{equation}\label{HermDkn}
h^{(k)}_n(t)=2^{k/2}\sqrt{n(n-1)\cdots(n-k+1)}\,h_{n-k}(t),\quad k\le n.
\end{equation}

{\bf Generalized Laguerre polynomials.}  In this case, we first easily can obtain that
\begin{equation}\label{GLONorP0}
L_n^s(t)=(-1)^n\sqrt{\frac{\Gamma(n+s+1)}{n!}}\,\ell_n^s(t),
\end{equation}
Now we start with the
known expansion \cite[p.~356]{AAR1999}
\[L_n^\beta(t)=\sum_{\nu=0}^n\frac{(\beta-\alpha)_{n-\nu}}{(n-\nu)!}L_\nu^\alpha(t),\]
where $(\alpha)_k$ denotes the Pochhammer symbol (or the shifted factorial, since $(1)_k=k!$) defined for any complex number $\alpha$ by
\[(\alpha)_k=\frac{\Gamma(\alpha+k)}{\Gamma(\alpha)}=\left\{\begin{array}{ll}
1,&k=0,\\[1mm]
\alpha(\alpha+1)\cdots(\alpha+k-1),&k\in\mathbb{N}=\{1,2,\ldots\}.	
\end{array}\right.\]
Iterating the formula  (\ref{izvPL}) gives the corresponding formula for the $k$-th derivative
\[\frac{\D^k}{{\D}t^k}L_n^{s}(t)=(-1)^k L_{n-k}^{s+k}(t),\quad k\le n.\]
Using the last two formulas, with $\alpha=s$ and $\beta=s+k$, we  get
\[\frac{\D^k}{{\D}t^k}L_n^{s}(t)=(-1)^k \sum_{\nu=0}^{n-k}\frac{(k)_{n-k-\nu}}{(n-k-\nu)!}L_\nu^s(t),\quad k\le n.\]
Finally, by (\ref{GLONorP0}), we obtain the corresponding expansion in orthonormal generalized Laguerre polynomials,
\begin{equation}\label{expGLPon}
\frac{\D^k}{{\D}t^k}\ell_n^{s}(t)=\sum_{\nu=0}^{n-k} c_{\nu,n}^{(k)}(s)\ell_\nu^s(t),\quad k\le n,	
\end{equation}
where
\[c_{\nu,n}^{(k)}(s)=(-1)^{n-k-\nu}\sqrt{\frac{(\nu+1)_{n-\nu}}{(\nu+s+1)_{n-\nu}}}\ \binom{n-\nu-1}{k-1},	
\]
In the standard Laguerre  case $(s=0)$ we have
\[c_{\nu,n}^{(k)}=c_{\nu,n}^{(k)}(0)=(-1)^{n-k-\nu}\binom{n-\nu-1}{k-1}\]
and (\ref{expGLPon}) becomes
\begin{equation}\label{expLPons0}
\frac{\D^k}{{\D}t^k}\ell_n(t)=(-1)^{n-k}\sum_{\nu=0}^{n-k}(-1)^{\nu} \binom{n-\nu-1}{k-1}\,\ell_\nu(t),\quad k\le n.	
\end{equation}

{\bf Jacobi polynomials.} In order to get an analogous formula of (\ref{expGLPon}) for the Jacobi polynomials, we use the following expansion (cf. \cite[Lemma 7.1.1, p.~357]{AAR1999}), but written for the monic Jacobi polynomials,
\begin{equation}\label{forPngd}
\widehat P_n^{(\gamma,\delta)}(t)=\sum_{\nu=0}^n c_{\nu,n}(\alpha,\beta;\gamma,\delta)\widehat P_\nu^{(\alpha,\beta)}(t),	
\end{equation} 
where
\begin{align}\label{cvnabgd} 
c_{\nu,n}(\alpha,\beta;\gamma,\delta)&= 
 \frac{2^{n-\nu } \binom{n}{\nu } (\gamma +\nu +1)_{n-\nu }}{(n+\gamma +\delta +\nu +1)_{n-\nu }}\\ 
&\times{}_3F_2\left[\begin{array}{c}
-n+\nu,\ n+\nu+\gamma+\delta+1,\ \nu+\alpha+1\\[1mm]
\nu+\gamma+1,\ 2\nu+\alpha+\beta+2	
\end{array}\biggm|1\right].	\nonumber
\end{align}
Here  ${}_3F_2$ is the generalized hypergeometric function, which is, in general,  defined by (cf. \cite[Chap.~2]{AAR1999})
\[
 _pF_q \left[\begin{array}{c}
a_1,\,\ldots,\,a_p \\[2mm]
 b_1,\,\ldots,\,b_q 	
\end{array}\!\!\Biggm| z \right] = \sum_{n=0}^{+\infty}\frac{(a_1)_n \cdots (a_p)_n}{(b_1)_n \cdots (b_q)_n}\cdot \frac {z^n}{n!}.\]             
The function ${}_pF_q$ is implemented as {\tt HypergeometricPFQ} in  Wolfram's {\sc Mathematica} and suitable for both symbolic and numerical calculation.

Iterating (\ref{izvPJ}), again for the monic Jacobi polynomials, we obtain
\[\frac{{\D}^k}{{\D}t^k}\widehat P_n^{(\alpha,\beta)}(t)=n(n-1)\cdots(n-k+1)\widehat P_{n-k}^{(\alpha+k,\beta+k)}(t),\quad k\le n.\]
Now, taking $\gamma=\alpha+k$ and $\delta=\beta+k$ in (\ref{forPngd}), with $n:=n-k$,  we get
\[\frac{{\D}^k}{{\D}t^k}\widehat P_n^{(\alpha,\beta)}(t)=k!\binom{n}{k}\sum_{\nu=0}^{n-k} c_{\nu,n-k}(\alpha,\beta;\alpha+k,\beta+k)\widehat P_\nu^{(\alpha,\beta)}(t),\quad k\le n.\]
Finally, for the orthonormal Jacobi polynomials $p_\nu^{(\alpha,\beta)}(t)=\gamma_\nu \widehat P_n^{(\alpha,\beta)}(t)$, we obtain
\begin{equation}\label{expanDkJ}
\frac{{\D}^k}{{\D}t^k}p_n^{(\alpha,\beta)}(t)=\sum_{\nu=0}^{n-k} d_{\nu,n}^{(k)}(\alpha,\beta)p_\nu^{(\alpha,\beta)}(t),\quad k\le n,	
\end{equation}
where
\begin{equation}\label{deoverce}
d_{\nu,n}^{(k)}(\alpha,\beta)=k!\binom{n}{k}\frac{\gamma_n}{\gamma_\nu} c_{\nu,n-k}(\alpha,\beta;\alpha+k,\beta+k)	
\end{equation}
and $\gamma_n=\gamma_n(\alpha,\beta)$ is defined by 
\begin{equation}\label{lcJpOrtnor0}
p^{(\alpha,\beta)}_0 (t)=\gamma_0=\frac{1}{\sqrt{\mu_0}}=\sqrt{\frac{\Gamma(\alpha+\beta+2)}{2^{\alpha+\beta+1}\Gamma(\alpha+1)\Gamma(\beta+1)}}\,.
\end{equation}
The leading coefficients $\gamma_n$ of the orthonormal Jacobi polynomial $p^{(\alpha,\beta)}_n (t)=\gamma_n t^n+\mbox{terms of lower degree}$,\  $n\ge1$, are given by (cf. \cite[p.~133]{gmgvm2008})
\begin{equation}\label{lcJpOrtnor}
\gamma_n=
   \frac{\sqrt{2n+\alpha+\beta+1}\ \Gamma(2n+\alpha+\beta+1)}{\sqrt{2^{2n+\alpha+\beta+1}n!\,\Gamma(n+\alpha+1)\Gamma(n+\beta+1)\Gamma(n+\alpha+\beta+1)}}. 	
\end{equation}

In the simplest (Legendre) case, when $w(t)=1$ ($\alpha=\beta=0$), the  leading coefficient (\ref{lcJpOrtnor}) reduces to
\[\gamma_n=\frac{1}{2^n}\sqrt{\frac{2n+1}{2}}\,\binom{2n}{n}.\]

In four Chebyshev cases when $\alpha,\beta\in\{-1/2,1/2\}$, we have the following coefficients for the monic polynomials:
\smallskip

For the Chebyshev weight of the first kind $w(t)=1/\sqrt{1-t^2}$ \ $(\alpha=\beta=-1/2)$:
$\alpha_n=0\ (n\ge0);\  \beta_0=\pi,\ \beta_1={1}/{2}, \ \beta_n={1}/{4}\  (n\ge2)$;
\smallskip

For the Chebyshev weight of the second kind $w(t)=\sqrt{1-t^2}$ \ $(\alpha=\beta=1/2)$:  
$\alpha_n=0\ (n\ge0);\  \beta_0= {\pi}/{2},\  \beta_n= {1}/{4}\ (n\ge1)$;
\smallskip 

For the Chebyshev weight of the third kind $w(t)=\sqrt{(1+t)/(1-t)}$ 
\ $(\alpha=-\beta=-1/2)$: 
$\alpha_0= {1}/{2},\  \alpha_n=0\  (n\ge1);\ \beta_0=\pi,\  \beta_n= {1}/{4}\ (n\ge1)$;
\smallskip

For the Chebyshev weight of the fourth kind  $w(t)=\sqrt{(1-t)/(1+t)}$  \ $(\alpha=-\beta=1/2)$: 
$\alpha_0=-{1}/{2},\  \alpha_n=0\ (n\ge1);\  \beta_0=\pi,\ \beta_n= {1}/{4}\  (n\ge1)$.
\smallskip

The  leading coefficients in the corresponding orthonormal polynomials for four Chebyshev cases are
\begin{align*}
(\alpha=\beta=-1/2):\qquad &\gamma_0=\frac{1}{\sqrt{\pi}},\quad  \gamma_n=2^{n-1}\sqrt{\frac{2}{\pi}}\quad (n\ge 1);\\[2mm] 
(\alpha=\beta=1/2):\qquad &\gamma_n=2^n\sqrt{\frac{2}{\pi}}\quad (n\ge 0);\\[2mm]
(\alpha=-\beta=\pm1/2):\qquad &\gamma_n=\frac{2^n}{\sqrt{\pi}}\quad (n\ge 0).
\end{align*}

\section{Extremal Problems of Markov's Type for Polynomials in  $L^2$--Norms for the Classical Weight Functions}\label{SEC4}

In this section we consider Markov's $L^2$-extremal problem \eqref{MEProb} on the space $X=L^2[(a,b);w]$, i.e.,
\begin{equation}\label{MEProb1}
M_{n,k}=\sup_{p\in\PP_n\setminus\{0\}}\frac{{\|p^{(k)}\|}_X}{{\|p\|}_X}
\quad(1\le k\le n),	
\end{equation}
 with the inner product defined by (\ref{innPrCOP}), where $w$ is a classical weight function ($w\in CW$). 

In 1987  Milovanovi\'c \cite{RJ50}  showed that the exact constant in (\ref{MEProb1}) can be found as the maximal eigenvalue of a square matrix of Gram's type or as the spectral norm of one triangular matrix.

Here we consider this extremal problem on  $X=L^2[(a,b);w]$, with the inner product defined by (\ref{innPrCOP}), where  $w\in CW$. As before,  the corresponding classical orthonormal polynomials will be denoted by $q_n(t)$, $n=0,1,\ldots\,$ .  The main approach will be based on analysis of the differentiation operator $D^k\equiv{\D}^k/{{\D}t^k}$, as a linear map $D^k:\PP_n\to\PP_{n-k}$ and its matrix of type
$(n-k+1)\times n$, using the orthonormal basis ${\Cale B}_m=
\{q_0,q_1,\ldots,q_m\}$ in the finite-dimensional spaces $\PP_m$, with $m=n$ and $m=n-k$.

Let $p(t)$ be an arbitrary real polynomial  in $\PP_n$, which can be uniquely represented in the orthonormal basis ${\Cale B}_n$ as
\begin{equation}\label{pinBn}
p(t)=\sum_{j=0}^nc_jq_j(t)\quad (c_j\in\mathbb{R},\ j=0,1,\ldots,n).	
\end{equation} 
Then the differentiation operator $D^k$, which maps elements (polynomials) from the $(n+1)$-dimensional space $\PP_n$ to another $(n-k+1)$-dimensional space $\PP_{n-k}$, can be uniquely described by the images of all basis polynomials $q_\nu(t)$,  $\nu=0,1,\ldots,n$, in the space $\PP_{n-k}$, represented in the corresponding basis of orthonormal polynomials ${\Cale B}_{n-k}=\{q_0,q_1,\ldots,q_{n-k}\}$. 

Since $\deg q_j(t)=j$, we have 
\begin{equation}\label{razv}
\left\{\begin{array}{rcl}	
D^kq_j(t)&=&0,\qquad 0\le j\le k-1,\\[1mm]
D^kq_k(t)&=&a_{1,k}q_0(t),\\[1mm]
D^kq_{k+1}(t)&=&a_{1,k+1}q_0(t)+a_{2,k+1}q_1(t),\\
&\vdots&\\
D^kq_{j}(t)&=&a_{1,j}q_0(t)+a_{2,j}q_1(t)+\cdots+a_{j-k+1,j}q_{j-k}(t),\\
&\vdots&\\
D^kq_{n}(t)&=&a_{1,n}q_0(t)+a_{2,n}q_1(t)+\cdots+a_{n-k+1,n}q_{n-k}(t),	
\end{array}\right.
\end{equation}
where
\begin{equation}\label{elemAnk}
a_{i,j}= (D^kq_j,q_{i-1})_w=\int_a^b q_j^{(k)}(t)\,{q_{i-1}(t)}\,w(t)\,{\D}t,
\end{equation}
so that the  matrix of the operator $D^k:\PP_n\to\PP_{n-k}$
 is given by
\begin{equation}\label{matrOAnk}
\left[\begin{array}{ccccccccc}
0&0&\cdots&0&|&a_{1,k}&a_{1,k+1}&\cdots &a_{1,n}\\[1mm]
0&0&\cdots&0&|&0&a_{2,k+1}&\cdots &a_{2,n}\\	
\vdots&\vdots&&\vdots&|&\vdots&\vdots&\ddots &\vdots\\[1mm]
0&0&\cdots&0&|&0&0&\cdots &a_{n-k+1,n}		
\end{array}\right]=\bigl[\,\mathbf{O}\,\,|\,\,A_{n,k}\bigr],	
\end{equation}
where $\mathbf{O}$ is a zero matrix of the type $(n-k+1)\times n$, and $A_{n,k}$
is a real quadratic matrix of order $n-k+1$. The columns in the matrix (\ref{matrOAnk}) are vectors whose coordinates are coefficients in the expansions of the images of $D^kq_j(t)$, $j=0,1,\ldots,n$, in the basis ${\Cale B}_{n-k}$.   
The quadratic matrix $A_{n,k}$ in (\ref{matrOAnk}) can be interpreted as a matrix of the differentiation operator $D^k$, denoted now as ${\Cale D}^k$, which maps the space $\PP_n\setminus\PP_{k-1}$ to $\PP_{n-k}$. Both of these spaces are of the same dimensions, i.e.,  $\dim\,(\PP_n\setminus\PP_{k-1})=\dim \PP_{n-k}=n-k+1$. Here, ${\Cale D}^k$ is a restriction of the operator $D^k$ from $\PP_n$ to $\PP_n\setminus\PP_{k-1}$. Thus, ${\Cale D}^k$ belongs
to the set of all linear operators from $\PP_n\setminus\PP_{k-1}$ to $\PP_{n-k}$, denoted by $L(\PP_n\setminus\PP_{k-1},\PP_{n-k})$. If we denote by ${\Cale M}_m$
the set of all quadratic matrices of order $m$, then the mapping 
${\Cale D}^k \mapsto A_{n,k}$ is a bijection from $L(\PP_n\setminus\PP_{k-1},\PP_{n-k})$ onto ${\Cale M}_{n-k+1}$ (cf.\ \cite[Chap.~II]{gvmrzdjLA}), and the extremal problem (\ref{MEProb1}) can be considered by methods of linear algebra. 

Using (\ref{pinBn}), for the norm of  each $p(t)\in\PP_n$ we have
\[ {\|p\|}_X^2=\sum_{j=0}^n c_j^2\ge \sum_{j=k}^n c_j^2, \]
with equality if and only if $c_0=c_1=\cdots=c_{k-1}$, 
as well as
\[
{\Cale D}^kp(t)=\sum_{j=k}^n c_j{\Cale D}^k q_j(t)=\sum_{j=k}^n c_j\left(\sum_{i=1}^{j-k+1} a_{i,j}q_{i-1}(t)\right),\]
i.e.,
\[
{\Cale D}^kp(t)=\sum_{i=1}^{n-k+1}\left(\sum_{j=i+k-1}^n a_{i,j}c_j\right)q_{i-1}(t).
\]
Since
\[{\|{\Cale D}^kp\|}_X^2=\sum_{i=1}^{n-k+1}\Biggl(\sum_{j=i+k-1}^n a_{i,j}c_j\Biggr)^2,\]
Markov's $L^2$-extremal problem (\ref{MEProb1}) can be considered in the form
\[M_{n,k}^2=\sup_{p\in\PP_n\setminus\{0\}}\frac{\|D^kp\|_X^2}{\|p\|_X^2}=
\sup_{p\in\PP_n\setminus\PP_{k-1}}\frac{\|{\Cale D}^kp\|_X^2}{\|p\|_X^2},
\]
i.e.,
\begin{equation}\label{MEProbMOD}
M_{n,k}^2= 
\max_{c_j\in\mathbb{R}\,(k\le j\le n)}\frac{\DS\sum_{i=1}^{n-k+1}\Biggl(\sum_{j=i+k-1}^n a_{i,j}c_j\Biggr)^2}{\DS \sum_{j=k}^n c_j^2},	
\end{equation}
where $c_0=c_1=\cdots=c_{k-1}$. 
 
Introducing two real $(n-k+1)$-dimensional vectors $\mathbf{c}=[c_k\ c_{k+1}\ \cdots\ c_n]^T$ and $\mathbf{y}=[y_1\ y_{2}\ \cdots\ y_{n-k+1}]^T$, defined by $\mathbf{y}=A_{n,k}\mathbf{c}$, where $A_{n,k}$ is a block matrix in 
(\ref{matrOAnk}), i.e., 
\begin{equation}\label{Ank-mat}
A_{n,k}=\left[\begin{array}{cccc}
a_{1,k}&a_{1,k+1}&\cdots&a_{1,n}\\[2mm]
0      &a_{2,k+1}&\cdots&a_{2,n}\\[2mm]	
\vdots&\vdots&\ddots&\vdots\\[2mm]
0&0&\cdots&a_{n-k+1,n}
\end{array}\right],	
\end{equation}
we can express the coordinates of the vector $\mathbf{y}$ by the following system of linear equations
\begin{equation}\label{sysEqY-C}
\left\{\begin{array}{rl}
y_1&=a_{1,k}c_k+a_{1,k+1}c_{k+1}+\cdots+a_{1,n}c_n,\\[1mm]
y_2&=\phantom{a_{1,k}c_k+\ \ }a_{2,k+1}c_{k+1}+\cdots+a_{2,n}c_n,\\[1mm]
	&\ \,\vdots\\[1mm]
y_{n-k+1}&=\phantom{a_{1,k}c_k+a_{2,k+1}c_{k+1}+\quad\,}a_{n-k+1,n}c_n.
\end{array}\ \right.	
\end{equation}
In this way, (\ref{MEProbMOD}) becomes
\begin{equation}\label{MEProbFIN}
M_{n,k}^2=  
\max_{c_j\in\mathbb{R}\,(k\le j\le n)}\frac{\DS\sum_{i=1}^{n-k+1} y_i^2}{\DS \sum_{j=k}^n c_j^2}=\max_{\mathbf{c}}\frac{\langle \mathbf{y},\mathbf{y} \rangle}{\langle\mathbf{c},\mathbf{c} \rangle}=\max_{\mathbf{c}}\frac{\left\langle A_{n,k}\mathbf{c},A_{n,k}\mathbf{c} \right\rangle}{\langle\mathbf{c},\mathbf{c} \rangle},	
\end{equation}
where $\langle\,\cdot\,\,\cdot\,\rangle$ is the standard inner product in an $(n-k+1)$-dimensional space. 

Introducing the matrices $B_{n,k}$ and $C_{n,k}$ by
\[B_{n,k}=A_{n,k}^TA_{n,k}\quad\mbox{and}\quad C_{n,k}=\bigl(A_{n,k}A_{n,k}^T\bigr)^{-1}=\bigl(A_{n,k}^T\bigr)^{-1} A_{n,k}^{-1},\]
(\ref{MEProbFIN}) can be written as
\begin{equation}\label{MEProbMODfs}
M_{n,k}^2= \max_{\mathbf{c}}\frac{\left\langle B_{n,k}\mathbf{c},\mathbf{c} \right\rangle}{\langle\mathbf{c},\mathbf{c} \rangle}=\left(\min_{\mathbf{y}}\frac{\left\langle C_{n,k}\mathbf{y},\mathbf{y} \right\rangle}{\langle\mathbf{y},\mathbf{y} \rangle}\right)^{-1}.	
\end{equation}
These symmetric matrices $B_{n,k}$ and $C_{n,k}$ are  positive definite, so that the corresponding quadratic forms $\left\langle B_{n,k}\mathbf{c},\mathbf{c} \right\rangle=\mathbf{c}^T B_{n,k}\mathbf{c}$ and $\left\langle C_{n,k}\mathbf{y},\mathbf{y} \right\rangle=\mathbf{y}^T C_{n,k}\mathbf{y}$ are positive for all non-zero vectors $\mathbf{c}$ and $\mathbf{y}$.
All eigenvalues of such matrices are positive numbers.

Now, for determining the best constant $M_{n,k}$ in (\ref{MEProbMODfs}) we use the well know result on the bounds for the quadratic form of a positive definite matrix $G$,
\begin{equation}\label{pdmQF}
\lambda_{\min}(G){\|\mathbf{x}\|}^2\le \mathbf{x}^T G\mathbf{x}
\le \lambda_{\max}(G){\|\mathbf{x}\|}^2,
\end{equation}  
where $\lambda_{\min}(G)$ and $\lambda_{\max}(G)$ are the minimal and maximal eigenvalues of the matrix $G$ and $\|\mathbf{x}\|=\sqrt{\langle\mathbf{x},\mathbf{x}\rangle}$ is the standard Euklidean norm of the vector~$\mathbf{x}$. 

This is how we get to the main result.

\begin{theorem}\label{thm323}
The best constant $M_{n,k}$ for the Markov $L^2$-extremal problem $(\ref{MEProb1})$ is given by
\begin{equation}\label{MnkBPa}
M_{n,k}=\sqrt{\lambda_{\max}(B_{n,k})},	
\end{equation}
where $\lambda_{\max}(B_{n,k})$ is the maximal eigenvalue of the matrix  $B_{n,k}=A_{n,k}^TA_{n,k}$, and $A_{n,k}$ is the upper triangular matrix $(\ref{Ank-mat})$, with elements given by
$(\ref{elemAnk})$. An extremal polynomial is 
\begin{equation}\label{Expol}
p^*(t)=\sum_{\nu=k}^n c_k^*q_\nu(t),
\end{equation}
 where $\mathbf{c^*}=[c_k^*\ \, c_{k+1}^*\ \cdots\ \,c_n^*]^T$ is the eigenvector of the matrix $B_{n,k}$ corresponding to the maximal eigenvalue 
$\lambda_{\max}(B_{n,k})$.	

Alternatively, 
\begin{equation}\label{MnkBPb}
M_{n,k}=\frac{1}{\sqrt{\lambda_{\min}(C_{n,k})}},	
\end{equation}
where $\lambda_{\min}(C_{n,k})$ is the minimal eigenvalue of the matrix $C_{n,k}$ defined by  $C_{n,k}=\bigl(A_{n,k}^T\bigr)^{-1} A_{n,k}^{-1}$. 
\end{theorem}

\begin{proof} Applying (\ref{pdmQF}) with $G=B_{n,k}$ and $G=C_{n,k}$ we obtain (\ref{MnkBPa}) and (\ref{MnkBPb}), respectively. Since left (right) inequality in (\ref{pdmQF}) reduces to an equality if and only if $\mathbf{x}$ is the eigenvector corresponding to the minimal (maximal) eigenvalue of the matrix $G$, we obtain the extremal   polynomial $p^*(t)$ given by \eqref{Expol}.
\end{proof}

In the following sections, we analyze the the corresponding extremal problems with the Hermite, the generalized Laguerre and the Jacobi weights, including  results on algorithms for numerical computation of the best constants, as well as their lower and upper bounds, asymptotic behaviour, etc.~obtained by several authors (cf. Milovanovi\'c \cite{RJ50}, D\"{o}rfler \cite{Dorf90}, \cite{Dorf91a}, \cite{Dorf2002}, Draux and Kaliaguine \cite{Draux2000},  \cite{Draux2006}, B\"{o}ttcher and D\"{o}rfler \cite{BotDor2010,BotDor2010a}, Aptekarev at al. \cite{aptekar2000},  \cite{aptekar2015},  \cite{aptekar2018},  Nikolov and Shadrin \cite{GNAS2017,GNAS2019}, etc.). 

 The elements of the matrix $A_{n,k}$,  defined  in (\ref{Ank-mat}),  are practically given by  formulas derived in Section~\ref{SEC3} for all $w\in CW$.  

\section{Extremal Problem With the Hermite Weight}\label{SEC5} 

We consider now the Hermite weight $w(t)={\E}^{-t^2}$ on $(-\infty,+\infty)$.

\begin{theorem}\label{ESscmidtHighDer}
Let $X=L^2[(-\infty,+\infty);w]$, with the Hermite weight $w(t)={\E}^{-t^2}$. Then the best constant in $(\ref{MEProb1})$  is given by
\begin{equation}\label{BConstH}
M_{n,k}=2^{k/2}\sqrt{\frac{n!}{(n-k)!}},\quad 1\le k\le n,	
\end{equation} 
with the extremal polynomial $p^*(t)=cH_n(x)$, where $c$ is a non-zero constant. 	
\end{theorem}

\begin{proof} According to (\ref{HermDkn}),   $A_{n,k}$ is a diagonal matrix of order $n-k+1$, with diagonal elements 
\[a_{i,k+i-1}=2^{k/2}\sqrt{j(j-1)\cdots(j-k+1)}\quad (j=k+i-1),\]
so that the matrix $B_{n,k}=A_{n,k}^TA_{n,k}$ in (\ref{MEProbMODfs}) is also diagonal, with  
eigenvalues $\lambda(B_{n,k})=2^k (k+i-1)\cdots i$, \ $i=1,\ldots,n-k+1$. Thus, using (\ref{MnkBPa}), we get the best constant 
\begin{align*}
M_{n,k}&=\sqrt{\max_{1\le i\le n-k+1}2^k (k+i-1)(k+i-2)\cdots i}\\[1mm]
& =
2^{k/2}\sqrt{n(n-1)\cdots(n-k+1)},	
\end{align*}
i.e.,  
(\ref{BConstH}). The extremal polynomial is $p^*(t)=cH_n(x)$, where $c$ is a non-zero constant. 
\end{proof}

The best constant in  Markov's  extremal problem 
on $L^2[(-\infty,+\infty);{\E}^{-t^2}]$
for the first derivative $M_n=M_{n,1}=\sqrt{2n}$ was solved in 1944 by E.\ Schmidt \cite{Schm} and later by Tur\'an \cite{tur1960}.

\begin{remark}
Theorem \ref{ESscmidtHighDer} was also proved  by Shampine \cite{Shamp,Shamp65}, D\"orfler \cite{Dorf87},
Milovanovi\'c \cite{RJ50}, as well as by Guessab and Milovanovi\'c \cite{RJ68,RJ72} using different methods. For other results in this subject see \cite{AllalJAT94,AllalAMH95}. 
\end{remark}

\section{Extremal Problem With the Generalized Laguerre Weight}\label{SEC6}

For considering Markov's extremal problems (\ref{MEProb1}) on the space $X=L^2[(0,+\infty);w]$, with the generalized Laguerre weight $w(t)=t^s{\E}^{-t}$ \ $(s>-1)$ and the 
inner product defined by (\ref{innPrCOP}), we use the expansion (\ref{expGLPon}) for an arbitrary $s>-1$ and (\ref{expLPons0}) for $s=0$.

Thus, the elements (\ref{elemAnk}) of the matrix $A_{n,k}=A_{n,k}(s)$ can be calculated as
\[a_{i,j}=c_{i-1,j}^{(k)}(s),\quad i=1,\ldots,n-k+1;\ j=k,\ldots,n,\]
using the explicit expression for $c_{\nu,n}^{(k)}(s)$ given in (\ref{expGLPon}).   
The following  commands in Wolfram's {\sc Mathematica} 12.3 provide this calculation, where the upper triangular matrix (\ref{Ank-mat}) is denoted  by {\tt Ank[n,k,s]},
 \begin{verbatim}
  Ce[v_,n_,k_,s_] := (-1)^(n-k-v) Sqrt[Pochhammer[v+1,n-v]/
      Pochhammer[v+s+1,n-v]]]Binomial[n-v-1,k-1]; 
  Ank[n_,k_,s_] :=Table[If[i==1, Table[Ce[0,j,k,s],{j,k,n}], 
      Join[Table[0,{j,k,k+i-2}], 
      Table[Ce[i-1,j,k,s],{j,k+i-1,n}]]],{i,1,n-k+1}]; 	
 \end{verbatim} 
as well as the matrices $B_{n,k}=B_{n,k}(s)$ and/or
  $C_{n,k}=C_{n,k}(s)$, that appear in Theorem \ref{thm323}. For example, for  given $n=6$, $k=2$ and $s=1$, the matrix $C_{n,k}(s)$, its minimal eigenvalue and the best constant $M_{n,k}=M_{n,k}(s)$ can be also obtained by the following commands in {\sc Mathematica},
\begin{verbatim}
  n=6; k=2; s=1; 
  Aa=Ank[n,k,s]; AaI=Inverse[Aa]; Cnk=Transpose[AaI].AaI;
  {lamdamim}=Eigenvalues[N[Cnk,16],-1][[1]]; 
  Mnk=1/Sqrt[lambdamin]; 	
 \end{verbatim}  
i.e.,
\[A_{6,2}(1)=\left[
\begin{array}{ccccc}
 \frac{1}{\sqrt{3}} & -1 & \frac{3}{\sqrt{5}} & -2 \sqrt{\frac{2}{3}} \
& \frac{5}{\sqrt{7}} \\[2.5mm]
 0 & \frac{1}{\sqrt{2}} & -2 \sqrt{\frac{2}{5}} & \sqrt{3} & -4 \
\sqrt{\frac{2}{7}} \\[2.5mm]
 0 & 0 & \sqrt{\frac{3}{5}} & -\sqrt{2} & 3 \sqrt{\frac{3}{7}} \\[2.5mm]
 0 & 0 & 0 & \sqrt{\frac{2}{3}} & -\frac{4}{\sqrt{7}} \\[2.5mm]
 0 & 0 & 0 & 0 & \sqrt{\frac{5}{7}} 
\end{array}\right] \]
and
\[
C_{6,2}(1)=\left[
\begin{array}{ccccc}
 3 & 3 \sqrt{2} & \sqrt{3} & 0 & 0 \\[2mm]
 3 \sqrt{2} & 8 & 4 \sqrt{\frac{2}{3}}+\sqrt{6} & \sqrt{2} & 0 \\[2mm]
 \sqrt{3} & 4 \sqrt{\frac{2}{3}}+\sqrt{6} & 8 & 3 \sqrt{3} & \
\sqrt{\frac{5}{3}} \\[2mm]
 0 & \sqrt{2} & 3 \sqrt{3} & \frac{15}{2} & \
\frac{6}{\sqrt{5}}+\sqrt{5} \\[2mm]
 0 & 0 & \sqrt{\frac{5}{3}} & \frac{6}{\sqrt{5}}+\sqrt{5} & \
\frac{36}{5}
\end{array}\right],\]
as well as
\[\lambda_{\min}(C_{6,2}(1))=0.036487159752501  \ \ \mbox{and}\ \  M_{6,2}(1)=5.235160139118.\]
Note that  $C_{6,2}(1)$ is a symmetric five-diagonal matrix.

In 1960 P.\ Tur\'an \cite{tur1960} obtained the sharp constant for the first derivative $(k=1)$ in the explicit form for $s=0$, i.e., when $X=L^2[(0,+\infty);{\E}^{-t}]$:
\begin{equation}\label{TuranExtrConst}
M_{n,1}(0)=M_n(0)=\left(2\sin \frac\pi{4n+2}\right)^{-1},
\end{equation} 
as well as the extremal polynomial  
\begin{equation}\label{TuranExtrPol}
p^*(t)=\sum_{\nu=1}^n \sin\frac{\nu\pi}{2n+1}L_{\nu}(t),
\end{equation}
where $L_{\nu}$ is the Laguerre polynomial.	
Schmidt \cite{Schm} also considered this problem and  obtained that
\[M_n=\frac{2n+1}\pi\left(1-\frac{\pi^2}{24(2n+1)^2}+\frac{R}{(2n+1)^4}\right)^{-1},\]
where $-8/3<R<4/3$.

We note that in this simplest case $(k=1,\, s=0)$, according to (\ref{expLPons0}),  the matrices $A_{n,1}(0)$ and $C_{n,1}(0)$ are given by
\[
A_{n,1}(0)=\left[\begin{array}{cccc}
a_{1,1}&a_{1,2}&\cdots&a_{1,n}\\[2mm]
0      &a_{2,2}&\cdots&a_{2,n}\\[2mm]	
\vdots&\vdots&\ddots&\vdots\\[2mm]
0&0&\cdots&a_{n,n}
\end{array}\right]= \left[\begin{array}{crcc}
1&-1&\cdots&(-1)^{n-1}\\[2mm]
0      &1&\cdots&(-1)^n\,\\[2mm]	
\vdots&\vdots&\ddots&\vdots\\[2mm]
0&0&\cdots&1
\end{array}\right]	
\]
and 
\begin{equation}\label{Cn1Jac}
C_{n,1}(0)=\bigl(A_{n,1}(0)^T\bigr)^{-1} A_{n,1}(0)^{-1}=\left[
\begin{array}{crcrrr}
 1 &  1 &  0 & \cdots &  0 &  0\\[2mm]
 1 &  2 &  1 & \cdots &  0 &  0\\[1mm]
 0 & \, 1 &  2 & \ddots &  0 & 0\\[1mm]
\vdots &  \vdots & \ddots & \ddots &  1 & 0 \\[2mm]
 0 &  0 & 0 & 1 & 2 &  1 \\[2mm]
 0 &  0 &  0 & 0 & 1 &  2 \\
\end{array}\right].	
\end{equation}
The last tridiagonal matrix  can be interpreted as a Jacobi matrix $J_n(\widetilde w)$ given by (\ref{JacobiMAT}) for certain sequence of monic orthogonal polynomials $\pi_\nu(x)$, $\nu=0,1,\ldots,n-1$,  with respect to some weight function $x\mapsto \widetilde w(x)$.  These polynomials satisfy the three-term 
recurrence relation
\begin{equation}\label{eqRRel}
\pi_{\nu+1}(x)=(x-\alpha_\nu)\pi_{\nu}(x)-\beta_\nu\pi_{n-1}(x),\quad \nu=1,\ldots,n-1, 	
\end{equation}
with $\pi_0(x)=1$, $\pi_{-1}(x)=0$, where the coefficients  are
$\alpha_0=1$, $\alpha_\nu=2$, $\beta_\nu=1$, $\nu=1,\ldots,n-1$. Putting $2-x=2\cos\theta$, i.e., $x=4\sin^2(\theta/2)$, \eqref{eqRRel} is a linear difference equation of the second order, whose general solution is given by $\pi_\nu(x)=(-1)^\nu\bigl[C_1\cos\nu\theta+C_2\sin\nu\theta\bigr]$, where $C_1$ and $C_2$ are arbitrary constants. 

Since $\pi_0(x)=1$ and $\pi_1(x)=x-1=1-2\cos\theta$, we  
obtain the particular solution
\begin{equation}\label{Resenje}
\pi_\nu(x)=(-1)^\nu \left[\cos\nu\theta-\tan\frac{\theta}{2}\,\sin\nu\theta\right]=(-1)^\nu \frac{\cos\left(\nu+\frac12\right)\theta}{\cos\frac{1}{2}\theta},
\end{equation}
which enables us to find the eigenvalues of the matrix $C_{n,1}(0)$ as the zeros of the polynomial $\pi_n(x)$
in the explicit form,
\[x_\nu=4\sin^2\frac{\theta_\nu}{2}=4\sin^2\frac{(2\nu-1)\pi}{2(2n+1)},\quad \nu=1,\ldots,n.\]
Since the minimal eigenvalue is $x_1$, i.e.,
\[\lambda_{\min}(C_{n,1}(0))=4\sin^2\frac{\pi}{2(2n+1)},\]
according to (\ref{MnkBPb}), the best constant in Tur\'an's case is given by \eqref{TuranExtrConst}.

It is easy to check that the $p^*(t)$, given by  (\ref{TuranExtrPol}), is an extremal polynomial on $X=L^2[(0,+\infty);{\E}^{-t}]$.

\begin{remark}
The polynomials $\pi_\nu(x)$ in 	(\ref{Resenje}) can be expressed in terms of Chebyshev polynomials of the third kind $V_\nu$, which are  given by 
\[V_\nu(x)=\frac{\nu!}{\left(\frac12\right)_\nu}\,P_\nu^{(-1/2,1/2)}(x)=\frac{\cos\left(\nu+\frac12\right)\theta}{\cos\frac{1}{2}\theta}\]	
 Indeed, in this case, 
\[\cos\theta=1-2\sin^2\frac{\theta}{2}=1-\frac{x}{2},\]
and we have  
\[\pi_\nu(x)=(-1)^\nu V_\nu\left(1-\frac{x}{2}\right).\]
The corresponding weight function is $x\mapsto \widetilde w(x)=\sqrt{(4-x)/x}$, supported on $[0,4]$.
\end{remark}

Now we consider the extremal problem (\ref{MEProb1}) for 
$k=1$ on the space $X=L^2[(0,+\infty);w]$, with the generalized Laguerre weight $w(t)=t^s{\E}^{-t}$ \ $(s>-1)$. 
First, we define the sequence $\beta_i=1+s/i$, \ $i=1,\ldots,n$, so that the elements (\ref{elemAnk}) of the matrix $A_{n,1}(s)$ can be expressed in the form
\begin{equation}\label{aij1s}
a_{i,j}=c_{i-1,j}^{(1)}(s)=(-1)^{j-i}\sqrt{\frac{(i)_{j-i+1}}{(i+s)_{j-i+1}}}=\frac{(-1)^{j-i}}{\sqrt{\beta_i\,\beta_{i+1}\,\cdots \,\beta_j}}, \end{equation}
for each $i=1,2,\ldots,n$ and  $j=i,\ldots,n$, and we can 
formulate the following result:

\begin{theorem}\label{GenLags}
Let $X=L^2[(0,+\infty);w]$, with the generalized Laguerre weight function $w(t)=t^s{\E}^{-t}$, \ $s>-1$. Then the tridiagonal matrix $C_{n,1}$  is given by 
\begin{equation}\label{JacobiMATCn1s}
 C_{n,1}=C_{n,1}(s)=\left[\begin{array}{ccccc}
 \alpha_0&\sqrt{\beta_1}&&&{\mathbf{O}}\\[1mm]
  \sqrt{\beta_1}&\alpha_1&\sqrt{\beta_2}\\
 &\sqrt{\beta_2}&\alpha_2&\ddots  \\
  &&\ddots&\ddots&\sqrt{\beta_{n-1}}\\[1mm]
{\mathbf{O}}&&&\sqrt{\beta_{n-1}}&\alpha_{n-1}
\end{array}\right],
 \end{equation}
where
\[\alpha_0=1+s,\quad \alpha_\nu=2+\frac{s}{\nu+1},\quad \beta_\nu=1+\frac{s}{\nu},\quad \nu=1,\ldots,n-1,\]
and the corresponding monic orthogonal polynomials $\pi_\nu(x)$, $\nu=0,1,\ldots$, satisfy the three term recurrence relation
\begin{align}\label{eqRRelRep}
\pi_{\nu+1}(x)&=(x-\alpha_\nu)\pi_\nu(x)	-\beta_\nu\pi_{\nu-1}(x),\quad \nu=0,1,\ldots,\\
\pi_0(x)&=1,\ \ \pi_{-1}=0.\nonumber
\end{align}
The best constant is given by
$$M_{n}=M_n(s)=\frac{1}{\sqrt{\lambda_{\min}(C_{n,1}(s))}}=\frac{1}{\sqrt{x_1}},$$
 where $x_1=x_1^{(n)}$ is minimal zero of the polynomial $\pi_n(x)$. 	
\end{theorem}

\begin{proof} According to \eqref{aij1s} and the system of equations 
\eqref{sysEqY-C} for $k=1$, we have
\[y_i=\sum_{j=i}^n \frac{(-1)^{j-i}}{\sqrt{\beta_i\,\beta_{i+1}\,\cdots \,\beta_j}},\quad i=1,\ldots,n.\]
Putting $y_{n+1}=0$, we see that for each $i=1,\ldots,n$,
\[
\sqrt{\beta_i}\,y_i+y_{i+1}=\sqrt{\beta_i}\sum_{j=i}^n \frac{(-1)^{j-i}}{\sqrt{\beta_i\,\beta_{i+1}\,\cdots \,\beta_j}}\,c_j-
\sum_{j=i+1}^n \frac{(-1)^{j-i}}{\sqrt{\beta_{i+1}\,\cdots \,\beta_j}}\,c_j=c_i.\] 

These equations  provide the unique solution of
the upper triangular system of equations (\ref{sysEqY-C}) in the matrix form  $\mathbf{c}=A_{n,1}^{-1}\mathbf{y}$, where
\[A_{n,1}^{-1}=\left[\begin{array}{ccccc}
\sqrt{\beta_1}&1&&&{\mathbf{O}}\\[1mm]
 & \sqrt{\beta_2}&1& \\
 &&\ddots&\ddots  \\[1mm]
  &&&\sqrt{\beta_{n-1}}&1\\[1mm]
{\mathbf{O}}&&&&\sqrt{\beta_{n}}
\end{array}\right].
 \]
 
Using this two-diagonal matrix and its transposed matrix $\bigl(A_{n,1}^{-1}\bigr)^T=\bigl(A^T_{n,1}\bigr)^{-1}$ we get the tridiagonal matrix $C_{n,1}=\bigl(A_{n,1}^T\bigr)^{-1} A_{n,1}^{-1}$ given by (\ref{JacobiMATCn1s}), where 
\[\alpha_0=\beta_1=1+s\quad\mbox{and}\quad  \alpha_\nu=1+\beta_{\nu+1}=2+\frac{s}{\nu+1},\quad \nu=1,\ldots,n-1.\] 

As we mentioned before, the  
tridiagonal matrix $C_{n,1}$ can be interpreted as a Jacobi matrix  for certain sequence of monic orthogonal polynomials $\pi_\nu(x)$, $\nu=0,1,\ldots,n-1$, which satisfy the three-term 
recurrence relation (\ref{eqRRelRep}). Note that for $s=0$, (\ref{JacobiMATCn1s}) reduces to (\ref{Cn1Jac}).

The zeros of $\pi_n(x)$ are mutually different real numbers and they  coincide with the eigenvalues of the Jacobi matrix \eqref{JacobiMAT}, in this case the matrix $C_{n,1}$.
The final conclusion of this theorem follows directly from 	Theorem \ref{thm323}.
\end{proof}

\begin{remark}
For $n=1$ and $n=2$, the best constants are
\[M_{1}(s)=\frac{1}{\sqrt{s+1}}\quad\mbox{and}\quad 
M_2(s)=\sqrt{\frac{3 (s+2)+\sqrt{(s+2) (s+10)}}{2(s+1) (s+2)}}.\]	
\end{remark}

According to Theorem \ref{GenLags}, an efficient software package {\tt OrthogonalPoly\-nomials} (see \cite{CveMil2004,MilCve2012}) in Wolfram's {\sc Mathematica}, developed for constructing   orthogonal polynomials and quadrature formulas of Gaussian type, can be used very easy for finding the best constants $M_n=M_{n}(s)$.

For example, if we want to calculate the best constants $M_{n}(s)$ for arbitrary $n\le 200$ we need the following sequence of commands:
\begin{verbatim}
  << orthogonalPolynomials`
     alpha[s_]:=Table[If[k==0,1+s,2+s/(k+1)], {k,0,199}]; 
     beta[s_]:=Table[If[k==0,1,1+s/k], {k,0,199}];
     Minzero[n_,s_]:=aZero[n,alpha[s],beta[s],
           WorkingPrecision->30,Precision->25][[1]];
     Mn[n_,s_]:=1/Sqrt[Minzero[n,s]]	
\end{verbatim}

Note that we put $\beta_0=1$ ({\tt beta[s][[1]]=1}), which is not important (see a comment in Section \ref{SEC2}). The routine {\tt aZero} computes all zeros of an orthogonal polynomial of degree $n$ as an increasing sequence. Graphics of $s\mapsto M_n(s)$ for
$n=10,50,100,150$  and $s\in(-1,4)$ are displayed in Figure~\ref{slAn}. 
\begin{figure}[h]
 \center
\includegraphics[width=0.75\textwidth]{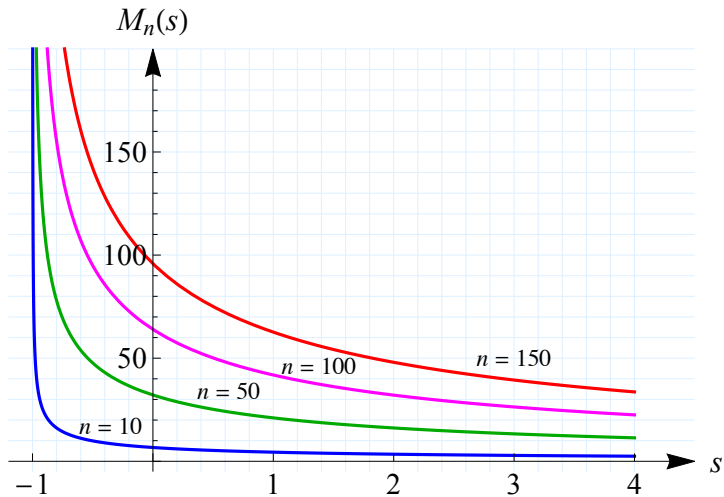}%
\caption{Best constants $M_n=M_n(s)$ for $n=10,50,100$ and $150$, when $s\in(-1,4)$}\label{slAn}
\end{figure}

There are several papers in which the authors give different estimates for the best constant $M_n(s)$ for the generalized Laguerre weight function. 

For example,  D\"orfler \cite{Dorf91a} determined certain lower and upper bounds for the best constant $M_n=M_n(s)$ in the form
\[\frac{1}{s+3}\left[\frac{n^2}{s+1}+\frac{(2s^2+5s+6)n}{3(s+1)(s+2)}+\frac{s+6}{3(s+2)}\right]	\le M_n(s)^2\le \frac{n(n+1)}{2(s+1)}
\]
for each $s>-1$, from which he obtained
\[\frac{1}{\sqrt{(s+1)(s+3)}}\le\varliminf_{n\to+\infty}\frac{M_{n}(s)}{n}\le\varlimsup_{n\to+\infty}\frac{M_n(s)}{n}\le
\frac{1}{\sqrt{2(s+1)}}.\]

In \cite{Dorf2002} D\"orfler proved that
\[M(s)=\lim_{n\to+\infty}\frac{M_n(s)}{n}=\frac{1}{j_{(s-1)/2,1}},\]
where $j_{\nu,1}$is the first positive zero of the Bessel function $J_\nu(x)$ of the first kind and order $\nu$. For
$s=0$ it gives $M(0)=1/j_{-1/2,1}=2/\pi$.

As an interesting result we mention a connection between the polynomials $\pi_\nu(x)$ from Theorem \ref{GenLags} and the {\it Pollaczek polynomials} $P_\nu(x):=P_\nu^\lambda(x;a,b)$, defined  by the recurrence relation \cite[p.~ 84]{Chi1978},
\[(\nu+1)P_{\nu+1}(x)=2[(a+\lambda+\nu)x+b]P_\nu(x)-(\nu+2\lambda-1)P_{\nu-1}(x),\quad \nu\ge0,\]
with $P_0(x)=1$ and $P_{-1}(x)=0$, in the form
\[\pi_\nu(x)= (-1)^n\frac{\nu+1}{s}\,P_{\nu+1}^{s/2}(1-x/2;-s/2,s/2)\]
(see \cite[Lemma 1]{Dorf2002}). This choice $\lambda=-a=b=s/2$ of parameters in $P_\nu^\lambda(x;a,b)$ causes the corresponding Pollaczek polynomials to be no longer orthogonal. In that case  $\deg P_\nu(x)=\nu-1$ for $\nu\ge1$ (see  \cite[Remark 3]{Dorf2002}).

More general problem was considered by Aptekarev, Draux, and Toulya\-kov \cite{aptekar2002} for the so-called {\it co-recursive Pollaczek polynomials} $A_n(x)$, defined by
\begin{equation}\label{AptAn}
A_{\nu+1}(x)=\bigl[x-\bigl(a+b+bc_{\nu+1}\bigr)\bigr]A_{\nu}(x)-ab(1+c_{\nu})A_{\nu-1}(x), 
\end{equation}
with $A_0(x)=1$ and $A_1(x)=x-b(1+c)$, where $c_\nu=c/\nu$ $(\nu\ge1)$ and $a,b$ and $c+1$ are positive numbers with $c\ne 0$. They studied the measure of orthogonality of these polynomials $A_\nu(x)$ outside the interval supporting the absolute continuous part of the measure of orthogonality of the corresponding non-perturbed system, i.e.,  the case  $c = 0$, when  \eqref{AptAn} reduces to 
\[A_{\nu+1}(x) =\bigl[x- (a+b)\bigr]A_{\nu}(x)-abA_{\nu-1}(x),\quad \nu\ge 1,\]
with $A_0(x) =1$ and  $A_1(x)=x-b$, and the interval of orthogonality 
\[\Delta=[\delta_1,\delta_2]=\Bigl[\bigl(\sqrt{a}-\sqrt{b}\bigr)^2,\bigl(\sqrt{a}+\sqrt{b}\bigr)^2\Bigr].\]
In their interesting paper the authors obtained conditions on the parameters $a$, $b$ and $c$ in order that the measure of orthogonality for the polynomials $A_\nu(x)$ on the intervals $(-\infty,\delta_1)$
 and $(\delta_2,+\infty)$  possesses (i) no mass point, or (ii) a single mass point, or (iii) infinitely many (convergent) mass points. Moreover, they determined the explicit position of these mass points.
 In an important  particular case for extremal problems, when $a=b=1$ and $c=s>-1$,  the polynomials $A_\nu(x)$ reduce  to  the polynomials $\pi_\nu(x)$ from Theorem \ref{GenLags}. A discrete Markov-Bernstein inequality was also treated in \cite{aptekar2002}. Some results concerning bounds on the  smallest zero of the polynomial $A_n(x)$ and its asymptotic behavior  were  obtained in \cite{aptekar2000}. In   \cite{Draux2000} Draux used some numerical methods ({\it qd} algorithm, fixed point methods of increasing order, and Laguerre's method) in order to obtain estimations for the polynomial zeros, as well as   an improvement of the estimate for the best constants. 

Recently, Nikolov and Shadrin \cite{GNAS2017} proved the following estimates for all $s>-1$ and $n\ge 3$
\[\frac{2\left(n+\frac{2}{3}s\right)\left(n-\frac{1}{6}(s+1)\right)}{(s+1)(s+5)}\le M_n(s)^2\le
\frac{(n+1)\left(n+\frac{2}{5}(s+1)\right)}{(s+1)\left((s+3)(s+5)\right)^{1/3}},\]
 where for the left-hand side inequality it is additionally assumed that $n\ge (s+1)/{6}$.	
Nikolov and Shadrin also obtained bounds for the limit value $M(s)=(j_{(s-1)/2,1})^{-1}$  in the form
\[\sqrt{\frac{2}{(s+1)(s+5)}}<M(s)<\frac{2}{s+2\pi-2},\quad s>1.\] 

The extremal problem \eqref{MEProb1} for the second derivative of polynomials $(k=2)$ was considered in \cite{RJ50} for the standard Laguerre weight function $w(t)={\E}^{-t}$ on $(0,+\infty)$. In that case we have $a_{i,j}=(-1)^{j-i+1}(j-i)$, so that the matrix $A_{n,2}=A_{n,2}(0)$ given by \eqref{Ank-mat} of order $n-1$  becomes
\begin{equation}\label{Ank-matfor0}
A_{n,2}=\left[\begin{array}{rrrrcc}
1&-2&3&\cdots&(-1)^{n-1}(n-2)&(-1)^n(n-1)\\[2mm]
0      &1&-2&\cdots&(-1)^{n-2}(n-3)&(-1)^{n-1}(n-2)\\[2mm]	
0 &0 &1&\cdots&(-1)^{n-3}(n-4)&(-1)^{n-2}(n-3)\\[2mm]
\vdots&\vdots&\vdots&\ddots&\vdots&\vdots\\[2mm]
0&0&0&\cdots&1&-2\ \ \\[2mm]
0&0&0&\cdots&0&1
\end{array}\right],
\end{equation}
and system of equation \eqref{sysEqY-C} can be written in the form
\[y_i=\sum_{j=i+1}^n (j-i)\,c_j,\quad  i=1,2,\ldots,n-1.\]
Introducing $y_n=y_{n+1}=0$, it is easy to see that
\[y_{i+1}+2y_i+y_{i-1}=c_i,\quad i=2,\ldots,n,\]
so that we get the inverse matrix of $A_{n,2}$ as a triangular and tridiagonal matrix
\[A_{n,2}^{-1}=  \left[
\begin{array}{c c c cccc}
 1 & 2 & \ \ \ 1 \ &   &   &  & \mathbf{O}  \\[2mm]
 & 1 & \ \ 2 & \ \ 1 &   &   &   \\[1mm]
  &   & \ \ \ 1 \ & \ \ \ddots & \ \ddots &   & \\
  &   &   & \ \ \ddots &\ \ddots & \ 1 \ &   \\[1mm]
 & &   &  & 1 & \ 2 \ & \ 1  \\[1mm]
  &   &  &  &   & \ 1 \ & \  2  \\[1mm]
 \mathbf{O}& &  &  &  &  & \ 1  \\
\end{array}
\right].\]
Finally, for the matrix $C_{n,2}=
\bigl(A_{n,2}^{-1}\bigr)^T A_{n,2}^{-1}$ we get 
the following five diagonal symmetric matrix of order $n-1$,
\[C_{n,2}=C_{n,2}(0)=\left[\begin{array}{cccccccc}
1&2&1&&&&&{\mathbf{O}}\\[3pt] 
2&5&4&1\\[3pt]
1&4&6&4&1\\ [3pt]
&1&4&6&4&1\\ [3pt]
&&\ddots&\ddots&\ddots&\ddots&\ddots\\ [3pt]
&&&1&4&6&4&1\\ [3pt]
&&&&1&4&6&4\\ [3pt]
{\mathbf{O}}&&&&&1&4&6	
\end{array}\right].\] 

Thus, using the minimal eigenvalue of the matrix $C_{n,2}$, we obtain the best constant
$M_{n,2}=M_{n,2}(0)=\left(\lambda_{\min}(C_{n,2})\right)^{-1/2}$. The best  constants for $n=4(1)10$, $n=20(10)50$, and $n=100$ are presented in Table \ref{tab161}, rounded to ten decimal digits to save space. 
In a similar way, we can also calculate the best constants $M_{n,2}(s)$ for each $s>-1$. The corresponding numerical values for the best constants 
for $s=-1/2,1,2$ are also presented in Table~\ref{tab161}.  

\begin{table}[h]
\caption{Best constants $M_{n,2}(s)$ for $n=4(1)10$, $n=20(10)50$, and  $n=100$, when  $s=0,-1/2,1,2$}
\label{tab161}
\begin{center}
\begin{tabular}{rcccc}
\hline
$n$ & $M_{n,2}(0)$ &$M_{n,2}(-1/2)$ & $M_{n,2}(1)$   & $M_{n,2}(2)$ \\
 \hline
$4$ & $4.402678830$ & $7.599353945$ & $2.343203955$  & $1.563671829$  \\ 
$5$ & $6.961320817$ & $12.12575130$ & $3.649155810$ & $2.405498748$ \\
$6$ & $10.08929121$ & $17.67855617$ &   $5.235160139$ & $3.421631009$ \\
$7$ & $13.78631814$ & $24.25741326$ &  $7.100933576$ & $4.611742398$  \\
$8$ & $18.05229187$ & $31.86216277$ & $9.246350532$ & $5.975669428$ \\ 
$9$ & $22.88716102$ & $40.49272196$ & $11.67134678$ & $7.513319423$  \\
$10$ & $28.29089888$ & $50.14904364$ & $14.37588601$ & $9.224635319$ \\
$20$ & $113.6144212$ & $203.1257944$ & $56.79356117$ & $35.88477549$ \\ 
$30$ & $255.8207947$ & $458.6698911$ & $127.1588062$ & $79.89882644$ \\
$40$ & $454.9097831$ & $816.7805071$ & $225.4710076$ & $141.2653074$ \\ 
$50$ & $710.8813577$ & $1277.457485$ & $351.7300494$ & $219.9838941$  \\
$100$ & $2843.977883$ & $5119.336722$ & $1402.227069$ & $873.8560224$\\ 
\hline 
\end{tabular} 
\end{center}	
\end{table}

\begin{remark}\label{Remark 1.6.1} The last problem can be connected 
with extremal problems of Wirtin\-ger's type (see Milovanovi\'c et al.
\cite[p.~578]{mil1994}, as well as Mitrinovi\'c \cite[p.~150]{MitAI}, Fan et al. \cite{FanTauTodd}, G.~V. Milovanovi\'c and I.~\v Z. Milovanovi\'c \cite{MilMil1}). \end{remark}

For $n=2$ and $n=3$ we  have the exact values: 
\[M_{2,2}(s)=\sqrt{\frac{2}{(s+1)(s+2)}},\quad
M_{3,2}(s)=\sqrt{\frac{2\left(2 s+9+\sqrt{s^2+24 s+72}\,\right)}{(s+1) (s+2) (s+3)}}\,.	\]
 
Note that $M_{2,2}(0)=1$ and $M_{3,2}(0)=1+\sqrt{2}$.

Figure \ref{MN2C} shows graphics of the best constants
$s\mapsto M_{n,2}(s)$ for $n=10$ and $n=50$ in $\log$-scale.
\begin{figure}[h]
 \center
\includegraphics[width=0.7\textwidth]{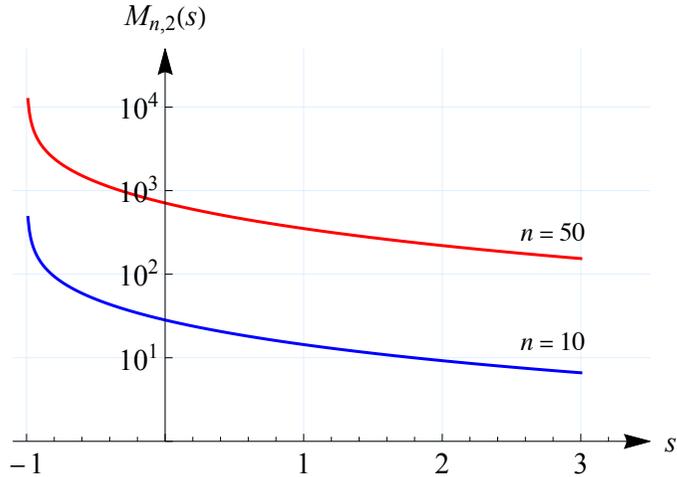} 
\caption{Best constants $s\mapsto M_{n,2}(s)$ in $\log$-scale for $n=10$ and $n=50$, when $s\in(-1,3)$}\label{MN2C}
\end{figure}

\section{Extremal Problem With the Jacobi Weight}\label{SEC7}

In this section we consider  Markov's extremal problem  (\ref{MEProb1}) on the space $X=L^2[(-1,1);(1-t)^\alpha(1+t)^\beta]$, where $\alpha,\beta>-1$. The 
inner product is defined by (\ref{innPrCOP}), and we use the expansion (\ref{expanDkJ}), jointly with the formulas \eqref{cvnabgd} and \eqref{deoverce}--\eqref{lcJpOrtnor}, in order to determine the elements of the matrix $A_{n,k}=A_{n,k}(\alpha,\beta)$. So we have
\[a_{i,j}=d_{i-1,j}^{(k)}(\alpha,\beta),\quad i=1,\ldots,n-k+1;\ j=k,\ldots,n,\]
where 
\[d_{\nu,n}^{(k)}(\alpha,\beta)=k!\binom{n}{k}\frac{\gamma_n}{\gamma_\nu} c_{\nu,n-k}(\alpha,\beta;\alpha+k,\beta+k),\]	
$c_{\nu,n}(\alpha,\beta;\gamma,\delta)$ is given by \eqref{cvnabgd}, and the  leading coefficients $\gamma_n$ of the orthonormal Jacobi polynomials are given by \eqref{lcJpOrtnor} for $n\ge 1$ and by \eqref{lcJpOrtnor0} for $n=0$.

The following  commands in {\sc Mathematica} 12.3 provide these calculations, where the upper triangular matrix (\ref{Ank-mat}) is denoted  by {\tt Ank[n,k,al,be]},
\begin{verbatim}
  c[v_,n_,al_,be_,ga_,de_]:=2^(n-v) Binomial[n,v] 
    Pochhammer[v+ga+1,n-v]/Pochhammer[n+v+ga+de+1,n-v] 
    HypergeometricPFQ[{-n+v,n+v+ga+de+1,v+al+1},
    {v+ga+1,2v+al+be+2},1];
  gammaJ[n_,al_,be_]:=If[n==0, 
    Sqrt[Gamma[al+be+2]/(2^(al+be+1)Gamma[al+1]Gamma[be+1])], 
    Sqrt[(2n+al+be+1)/(n!2^(2n+al+be+1))]Gamma[2n+al+be+1]/
    Sqrt[Gamma[n+al+1]Gamma[n+be+1]Gamma[n+al+be+1]]];
  De[v_,n_,k_,al_,be_]:=k! Binomial[n,k]gammaJ[n,al,be]
    c[v,n-k,al,be,al+k,be+k]/gammaJ[v,al,be]; 
  Ank[n_,k_,al_,be_]:=Table[If[i==1,
     Table[De[0,j,k,al,be],{j,k,n}], 
     Join[Table[0,{j,k,k+i-2}],Table[De[i-1,j,k,al,be],
     {j,k+i-1,n}]]],{i,1,n-k+1}];
\end{verbatim} 
as well as the matrices $B_{n,k}=B_{n,k}(\alpha,\beta)$ and/or
  $C_{n,k}=C_{n,k}(\alpha,\beta)$, that appear in Theorem \ref{thm323}. 
  
  For example, for  given $n=6$, $k=1$ and $\alpha=-\beta=-1/2$ (Chebyshev weight of the third kind), the matrix $C_{6,1}(-1/2,1/2)$, its minimal eigenvalue and the best constant $M_{6,1}=M_{6}(-1/2,1/2)$ can be also obtained by the following commands: 
\begin{verbatim}
  n=6; k=1; al=-1/2; be=1/2;
  Aa=Ank[n,k,al,be]; AaI=Inverse[Aa]; Cnk=Transpose[AaI].AaI;
  {lamdamim}=Eigenvalues[N[Cnk,16],-1][[1]]; 
  Mnk=1/Sqrt[lambdamin]; 	
 \end{verbatim}  
Thus,
\[C_{6,1}\biggl(-\frac{1}{2},\frac{1}{2}\biggr)=\left[
\begin{array}{cccccc}
 \frac{3}{8} & -\frac{1}{24} & -\frac{1}{24} & 0 & 0 & 0 \\[2.5mm]
 -\frac{1}{24} & \frac{7}{72} & -\frac{1}{144} & -\frac{1}{48} & 0 & 0 \\[2.5mm]
 -\frac{1}{24} & -\frac{1}{144} & \frac{13}{288} & -\frac{1}{480} & -\frac{1}{80} & 0 \\[2.5mm]
 0 & -\frac{1}{48} & -\frac{1}{480} & \frac{21}{800} & -\frac{1}{1200} & -\frac{1}{120} \\[2.5mm]
 0 & 0 & -\frac{1}{80} & -\frac{1}{1200} & \frac{37}{3600} & -\frac{1}{720} \\[2.5mm]
 0 & 0 & 0 & -\frac{1}{120} & -\frac{1}{720} & \frac{1}{144} 
\end{array}
\right],\]
and 
\[\lambda_{\min}(C_{6,1}(-1/2,1/2))=1.793818125493\times 10^{-3}, \ \ M_{6,1}=23.61080508655.\]
In the general case,  $C_{n,1}(\alpha,\beta)$ is a symmetric five-diagonal matrix.

\begin{remark}\label{rem32_1}
For the five-diagonal symmetric matrix $C_{n,1}(-1/2,1/2)=\left[c_{i,j}\right]_{i,j=1}^{n,n}$ of order $n$ (Chebyshev case of the third kind), one can prove that 
\[c_{i,i}=\left\{\begin{array}{ll}
\dfrac{i^2+i+1}{2i^2(i+1)^2},&i=1,\ldots,n-2,\\[4mm]
\dfrac{i^2+2i+2}{4i^2(i+1)^2},&i=n-1,\\[4mm]
\dfrac{1}{4i^2},&i=n;	
\end{array}\right.\]
\[\quad c_{i,i+1}=\left\{\begin{array}{ll}
\dfrac{-1}{2i(i+1)^2(i+2)},&i=1,\ldots,n-2,\\[4mm]
\dfrac{-1}{4i(i+1)^2},&i=n-1;
\end{array}\right.
	\]
and
\[c_{i,i+2}=-\frac{1}{4(i+1)(i+2)},\ \ n=1,\ldots,n-2.\]	
In the Chebyshev case of the fourth kind $(\alpha=-\beta=1/2)$, the all elements are the same as the previous ones, except $c_{i,i+1}$ which have only  opposite sign of the previous ones.
\end{remark}

In the sequel we consider Markov's extremal problems only with the Gegenbauer weight $(\alpha=\beta>-1)$ for $k=1$. The elements of the first sub-diagonals in the corresponding symmetric matrix 
\[C_{n}\equiv C_{n,1}(\alpha,\alpha)=\bigl[c_{i,j}\bigr]_{i,j=1}^{n,n}\quad (c_{j,i}=c_{i,j})\]
  are equal to zero and the eigenvalue problem for such a matrix  can be reduced to two eigenvalues problems for tridiagonal matrices (cf.\  \cite{RJ50}).

Thus, we now consider the matrix
\begin{equation}\label{matCnDek}
\arraycolsep=2pt\def\arraystretch{1.2}
C_{n}=\left[
\begin{array}{cccccccc}
c_{1,1}&0&c_{1,3}&&&&&\mathbf{O}\\ \vspace{4pt}
0&c_{2,2}&0&c_{2,4}\\ \vspace{4pt}
c_{3,1}&0&c_{3,3}&0&c_{3,5}\\ \vspace{4pt}
&c_{4,2}&0&c_{4,4}&0&c_{4,6}\\ \vspace{4pt}
&&\ddots&\ddots&\ddots&\ddots&\ddots\\ \vspace{4pt}
&&&c_{n-2,n-4}&0&c_{n-2,n-2}&0&c_{n-2,n}\\ \vspace{4pt}
&&&&c_{n-1,n-3}&0&c_{n-1,n-1}&0\\ \vspace{4pt}
\mathbf{O}&&&&&c_{n,n-2}&0&c_{n,n}
\end{array}\right], 
\end{equation}
and prove the  following auxiliary result which is related to a decomposition of determinants of this type of  matrices\footnote{For some more general cases, with the even weight function $t\mapsto w(t)$ on a symmetric interval $(-a,a)$, see  \cite{RJ50} and \cite[pp.~579--582]{mil1994}.}. 

\begin{lemma}\label{LemDecomp}
Let $E_m$ and $\widehat E_m$ be symmetric tridiagonal matrices given by
\begin{equation}\label{Em}
E_m=\left[\begin{array}{ccccc}
c_{1,1}&c_{1,3}&&&{\mathbf{O}}\\[1mm]
c_{3,1}&c_{3,3}&c_{3,5}\\[1mm]
 &c_{5,3}&c_{5,5}&\ddots  \\[1mm]
  &&\ddots&\ddots&c_{2m-3,2m-1}\\[1mm]
{\mathbf{O}}&&&c_{2m-1,2m-3}&c_{2m-1,2m-1}
\end{array}\right]
\end{equation}
and
\begin{equation}\label{hatEm}
\widehat E_m=\left[\begin{array}{ccccc}
c_{2,2}&c_{2,4}&&&{\mathbf{O}}\\[1mm]
c_{4,2}&c_{4,4}&c_{4,6}\\[1mm]
 &c_{6,4}&c_{6,6}&\ddots  \\[1mm]
  &&\ddots&\ddots&c_{2m-2,2m}\\[1mm]
{\mathbf{O}}&&&c_{2m,2m-2}&c_{2m,2m}
\end{array}\right].
\end{equation}
Then
\begin{equation}\label{proizDETS}
\det C_n=\det E_{[(n+1)/2]}\det\widehat E_{[n/2]}.	
\end{equation}
\end{lemma}
 
\begin{proof} For the determinant  of the matrix $C_n$, given by \eqref{matCnDek}, we use the Laplace expansion. 

Let first $n$ be even $(n=2m)$. Expanding by columns numbered $1$, $3$, $\ldots$, $n-1$, one finds that only one non-zero contribution results, namely from the minor and cominor pair
\[\left(\begin{array}{cccc}
1 & 3 & \cdots & n-1\\
1 & 3 & \cdots & n-1	
\end{array}\right),\quad 
\left(\begin{array}{cccc}
2 & 4 & \cdots & n\\
2 & 4 & \cdots & n	
\end{array}\right).\]
In this way, one  immediately obtains the relation \eqref{proizDETS}.

Similarly, Laplace expansion by columns $1, 3, \ldots, n$ gives the result for odd $n$.	
\end{proof}

Replacing $c_{i,i}$  by $c_{i,i}-\lambda$ from \eqref{proizDETS} we get
\[
\det\bigl[C_n-\lambda I_n\bigr]=\det\bigl[ E_{[(n+1)/2]}-\lambda I_{[(n+1)/2]}\bigr]\det\bigl[\widehat E_{[n/2]}-\lambda I_{[n/2]}\bigr],\]
where $I_m$ is an identity matrix of order $m$. In this way, the eigenvalue problem for $C_n$ reduces to two eigenvalue problems for matrices of lower orders and the second part of Theorem \ref{thm323} gives the following result:

\begin{theorem}\label{Modthm323Gegenbauer}
Let $X=L^2[(-1,1);(1-t^2)^\alpha]$, $\alpha>-1$, and the tridiagonal matrices $E_m$ and $\widehat E_m$ be  given as in Lemma  $\ref{LemDecomp}$.
Then the best constant $M_n$  in Markov's extremal problem is 
\begin{equation}\label{MnkBPbGBw}
M_{n}(\alpha)=\sup_{p\in\PP_n\setminus\{0\}}\frac{{\|p'}\|_X}{{\|p\|}_X}=\frac{1}{\sqrt{\min\bigl\{\lambda_{\min}(E_{[(n+1)/2]}),
\lambda_{\min}(\widehat E_{[n/2]})\bigr\}}},
\end{equation}
where $\lambda_{\min}(E_{[(n+1)/2]})$ and $\lambda_{\min}(\widehat E_{[n/2]})$ are the minimal eigenvalues of matrices $E_{[(n+1)/2]}$ and $\widehat E_{[n/2]}$, respectively.
\end{theorem}

In particular, we consider now  three different important cases (see \cite{RJ50}):\smallskip

{\bf 1.  Legendre case $(\alpha=\beta=0)$.} In this case, using the elements of the matrix $A_{n,1}^{-1}=\bigl[\widetilde a_{i,j}\bigr]_{i,j=1}^{n,n}$, with only two non-zero diagonals
\[\widetilde a_{i,i}=\frac{1}{\sqrt{4i^2-1}}\quad (i=1,\ldots,n)\]
and
\[\widetilde a_{i,i+2}=\frac{1}{\sqrt{(2i+1)(2i+3}}\quad(i=1,\ldots,n-2),\]
we get the elements of the matrix $C_n$ in \eqref{matCnDek} as
$C_n=(A_{n,1}^{-1})^TA_{n,1}^{-1}$, i.e.,
\[c_{i,i}=\left\{
\begin{array}{ll}
\dfrac{2}{(2i-1)(2i+3)},&i=1,\ldots,n-2,\\[4mm]
\dfrac{1}{4i^2-1},&i=n-1,\,n,
\end{array}\right.	\]
and
\begin{align*}
c_{i,i+1}&=0,\quad i=1,\ldots,n-1,\\[1mm]
c_{i,i+2}&=-\dfrac{1}{\sqrt{(2i+1)(2i+5)}},\quad  i=1,\ldots,n-2.	
\end{align*}

Since the elements $c_{i,i+2}$ are negative, in order to reduce the eigenvalue problems with the matrices 	\eqref{Em} and \eqref{hatEm} 
in the way we did before, we introduce 
\begin{align*}
\alpha_\nu&=\left\{\begin{array}{ll}
-\dfrac{2}{(4\nu+1)(4\nu+5)},&\nu=0,1,\ldots,\left[\dfrac{n+1}{2}\right]-2,\\[4mm]
-\dfrac{1}{(4\nu+1)(4\nu+3)},&\nu=\left[\dfrac{n+1}{2}\right]-1;
\end{array}\right.\\[3mm]
\beta_0 &=1,\ \ \ \beta_\nu = \dfrac{1}{(4\nu- 1) (4\nu + 1)^2 (4 \nu + 3)},\ \ \ 
\nu=1,\ldots,\left[\dfrac{n+1}{2}\right]-1;
\end{align*}
and
\begin{align*}
\widehat\alpha_\nu&=\left\{\begin{array}{ll}
-\dfrac{2}{(4\nu+3)(4\nu+7)},&\nu=0,1,\ldots,\left[\dfrac{n}{2}\right]-2,\\[4mm]
-\dfrac{1}{(4\nu+3)(4\nu+5)},&\nu=\left[\dfrac{n}{2}\right]-1;
\end{array}\right.\\[3mm]
\widehat\beta_0 &=1,\ \ \ \widehat\beta_\nu = \dfrac{1}{(4\nu+1) (4\nu + 3)^2 (4 \nu + 5)},\ \ \ 
\nu=1,\ldots,\left[\dfrac{n}{2}\right]-1.
\end{align*}

In fact, in this way  we consider the matrices $-C_n$, $-E_m$ and $-\widehat E_m$, which have their eigenvalues only with opposite signs from those for the original matrices $C_n$, $E_m$ and $\widehat E_m$, respectively. But now, the matrices $-E_m$ and $-\widehat E_m$ can be interpreted as Jacobi matrices for certain measures and we can, as before, to apply the efficient {\sc Mathematica}  package {\tt OrthogonalPolynomials} (see \cite{CveMil2004,MilCve2012})  for calculating the best constants $M_n=M_{n}(\alpha)$, given by \eqref{MnkBPbGBw}. 
\smallskip

{\bf 2. Chebyshev case of the first kind $(\alpha=\beta=-1/2)$.} In a similar way we obtain
\[c_{i,i}=\left\{\begin{array}{ll}
\dfrac{3}{4},&i=1, \\[4mm]
\dfrac{1}{2i^2},&i=2,\ldots,n-2,\\[4mm]
\dfrac{1}{4i^2},&i=n-1,\ n;	
\end{array}\right.	\]
and
\[c_{i,i+1}=0,\ \ i=1,\ldots,n-1;\quad c_{i,i+2}=-\dfrac{1}{4i(i+2)},\ \ i=1,\ldots,n-2,\]
so that the corresponding recurrence coefficients are
\begin{align*}
\alpha_{\nu}&=\left\{\begin{array}{ll}
-\dfrac{3}{4},&\nu=0, \\[4mm]
-\dfrac{1}{2(2\nu+1)^2},&\nu=1,\ldots,\left[\dfrac{n+1}{2}\right]-2,
\\[4mm]
-\dfrac{1}{4(2\nu+1)^2},&\nu=\left[\dfrac{n+1}{2}\right]-1;	\end{array}\right.	\\[3mm]
\beta_0&=1,\ \ \ \beta_\nu=\dfrac{1}{16(4\nu^2-1)^2},\quad \nu=1,\ldots,\left[\dfrac{n+1}{2}\right]-1;
\end{align*}
and
\begin{align*}
\widehat\alpha_{\nu}&=\left\{\begin{array}{ll}
-\dfrac{1}{8(\nu+1)^2},&\nu=0,1,\ldots,\left[\dfrac{n}{2}\right]-2,
\\[4mm]
-\dfrac{1}{16(\nu+1)^2},&\nu=\left[\dfrac{n}{2}\right]-1;	
\end{array}\right.	\\[3mm]
\widehat\beta_0&=1,\ \ \ \widehat\beta_\nu=\dfrac{1}{256\nu^2(\nu+1)^2},\quad \nu=1,\ldots,\left[\dfrac{n}{2}\right]-1.
\end{align*}
\medskip

{\bf 3. Chebyshev case of the second kind $(\alpha=\beta=1/2)$.} Here,
we have
\[c_{i,i}=\left\{\begin{array}{ll}
\dfrac{i^2+2i+2}{2i^2(i+2)^2},&i=1,\ldots,n-2,\\[4mm]
\dfrac{1}{4i^2},&i=n-1,\ n;	
\end{array}\right.	\]
and 
\[c_{i,i+1}=0,\ \ i=1,\ldots,n-1;\quad c_{i,i+2}=-\dfrac{1}{4(i+2)^2},\ \ i=1,\ldots,n-2,\]
so that  
\begin{align*}
\alpha_{\nu}&=\left\{\begin{array}{ll}
-\dfrac{4\nu^2+8 \nu+5}{2(2\nu+1)^2(2\nu+3)^2},&\nu=0,1,\ldots,	\Bigl[\dfrac{n+1}{2}\Bigr]-2,\\[3mm]
-\dfrac{1}{4(2\nu+1)^2},&\nu=\Bigl[\dfrac{n+1}{2}\Bigr]-1,
\end{array}\right.\\[3mm]
\beta_0&=1,\ \ \ \beta_\nu=\frac{1}{16(2\nu+1)^4},\ \ \nu=1,\ldots,\Bigl[\dfrac{n+1}{2}\Bigr]-1;
\end{align*}
and
\begin{align*}
\widehat\alpha_{\nu}&=\left\{\begin{array}{ll}
-\dfrac{2\nu^2+6\nu+5}{16(\nu+1)^2(\nu+2)^2},&\nu=0,1,\ldots,\left[\dfrac{n}{2}\right]-2,\quad\\[3mm]
-\dfrac{1}{16(\nu+1)^2},&\nu=\left[\dfrac{n}{2}\right]-1,
\end{array}\right.\\[3mm]	
\widehat\beta_0&=1,\ \ \ \widehat\beta_\nu=\frac{1}{256(\nu+1)^4},\quad\nu=1,2,\ldots,\left[\frac{n}{2}\right]-1.
\end{align*}

Figure \ref{MnCh012} shows graphics of the best constants
$n\mapsto M_{n}(\alpha)$ in \eqref{MnkBPbGBw} for $n=1(1)100$ in $\log$-scale for $\alpha=-1/2,0,1/2$.
\begin{figure}[h]
 \center
\includegraphics[width=0.82\textwidth]{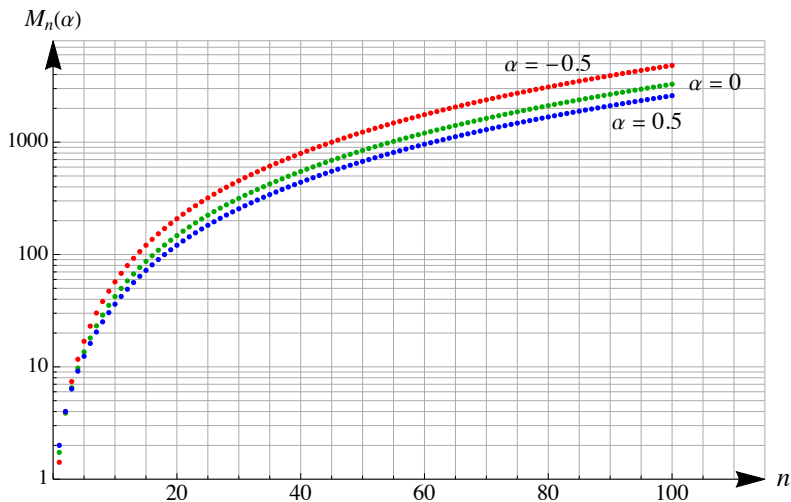} 
\caption{Best constants $n\mapsto M_{n}(\alpha)$ 
in \eqref{MnkBPbGBw} for $n=1(1)100$ in $\log$-scale, when $\alpha=-1/2,0,1/2$}\label{MnCh012}
\end{figure}

\begin{remark}\label{remEO}
In the previous cases $(\alpha=0$, $-1/2$, $1/2)$  the expression \eqref{MnkBPbGBw}  can be written as
\[
M_{n}(\alpha)=\left\{
\begin{array}{ll}
\dfrac{1}{\sqrt{\lambda_{\min}(E_{[(n+1)/2]})}},& \mbox{if $n$ is odd},\\[5mm]
\dfrac{1}{\sqrt{\lambda_{\min}(\widehat E_{[n/2]})}},& \mbox{if $n$ is even}.
\end{array}\right.
\] 
In addition see \cite[Theorem 2.3]{AleksovGeno2018}.	
\end{remark}

\begin{remark}
Recently Nikolov and Shadrin \cite{GNAS2019} (see also \cite{AleksovGenoSh2016,AleksovGeno2018,Draux2006}) 
 have been considered the best constant $M_n(\alpha)$ in \eqref{MnkBPbGBw}. Taking $\alpha=\lambda-1/2$, $\lambda>-1/2$, $\lambda'=\min\{0,\lambda\}$ and $\lambda'=\max\{0,\lambda\}$,  they derived explicit lower and upper bounds for $M_n(\alpha)=c_n(\lambda)$ for each $n\ge3$,
 \[\frac{1}4\frac{n^2(n+\lambda)^2}{(\lambda+1)(\lambda+2)}<[c_n(\lambda)]^2<\frac{n(n+2\lambda+2)^3}{(\lambda+2)(\lambda+3)},\quad\lambda\ge 2,\]
and
\[
\frac{(n+\lambda)^2(n+2\lambda')^2}{(2\lambda+1)(2\lambda+5)}<[c_n(\lambda)]^2<\frac{(n+\lambda+\lambda''+2)^4}{2(2\lambda+1)\sqrt{2\lambda+5}},\quad \lambda>-\frac{1}{2}.
\]	
\end{remark}

In \cite{aptekar2000} Aptekarev, Draux and Kalyagin discussed the asymptotics of $M_n(\alpha)$ in \eqref{MnkBPbGBw} when $n\to +\infty$, and proved that
the best constant has the asymptotics
\begin{equation}\label{eq1.2Apt}
M_n(\alpha)=\frac{n^2}{2j_{\nu(\alpha)}}\bigl[1+o(1)\bigr],\quad
\nu(\alpha)=\frac{\alpha-1}{2},	
\end{equation}
where $j_\nu$ is the smallest zero of the Bessel function $J_\nu(t)$.

Recently, Aptekarev, Draux, Kalyagin, and Tulyakov \cite{aptekar2015}
have considered the asymptotics of the best constant $M_n(\alpha,\beta)$
in the general extremal problem in the space $X=L^2[(-1,1);w]$, with the Jacobi weight function $w(t)=(1-t)^\alpha(1+t)^\beta$ \ $(\alpha,\beta>-1)$ on $(-1,1)$.

\begin{theorem}[\cite{aptekar2015}]\label{JacobiProbl}
Let the parameters of the Jacobi weight function $w(t)=(1-t)^\alpha(1+t)^\beta$ 	satisfy the restriction $|\alpha-\beta|<4$. Then, for the best constant
$M_n(\alpha,\beta)$  we have the asymptotics
\begin{equation}\label{eq1.5Apt}
M_n(\alpha,\beta)= \frac{n^2}{2j_{\nu^*}}\bigl[1+o(1)\bigr],\quad
\nu^*=\min\left\{\frac{\alpha-1}{2},\frac{\beta-1}{2}\right\},	
\end{equation}
where $j_{\nu^*}$ is the smallest zero of the Bessel function $J_{\nu^*}$.
\end{theorem}

Evidently for $\beta=\alpha$,  \eqref{eq1.5Apt} reduces to \eqref{eq1.2Apt}. In order to get the previous generalization, the authors needed to prove that linearly independent particular solutions of differential equations which satisfy the boundary conditions at the initial values of the discrete variable of a finite difference problem indeed are close to the particular solutions of the finite difference problem.

\begin{remark}
The authors of \cite{aptekar2015} pointed out that the most surprising result for them is the appearance of the restriction $|\alpha-\beta|<4$ on the parameters $\alpha,\beta$. Their comment was that they cannot prove or disapprove its necessity; however, they have to admit that this restriction is unavoidable in their proof strategy of Theorem \ref{JacobiProbl}.
\end{remark}

Recently Totik \cite{VT2019PAMS} has proved Theorem \ref{JacobiProbl}, without the previous restriction $|\alpha-\beta|<4$ on the parameters $\alpha,\beta$. Actually he proved a more general result, giving the exact asymptotic Markov constant for generalized Jacobi weight  on several intervals.

\begin{remark}
Using Sobolev spaces with continuous and discrete coherent pairs of weights, Aptekarev, Draux, and Tulyakov \cite{aptekar2018} studied the asymptotic behavior of the best constants in Markov-Bernstein inequalities with classical weighted integral norms.
\end{remark} 

\section{Markov's Extremal Problems on Some Restricted Classes of Polynomials}\label{SEC8}

In this section we only mention Markov's extremal problems 
on $X=L^2[(a,b);w]$ for some restricted classes of polynomials $W_n\subset\PP_n$, i.e.,
\begin{equation}\label{MEProb2}
M_{n,k}=\sup_{p\in W_n\setminus\{0\}}\frac{{\|p^{(k)}\|}_X}{{\|p\|}_X}
\quad(1\le k\le n).	
\end{equation}
Usually we can restrict zeros of polynomials or their coefficients. In this way, the corresponding best constant $M_{n,k}$  can be improved. For such kind of extremal problems in uniform norm see
Milovanovi\'c at al. \cite[pp.~624--643]{mil1994}.

In 1981 Varma \cite{Varma1981} investigated the problem of determining the best constant $C_n(s)=M_{n,1}(s)^2$  in
the inequality
\[{\|P'\|}_X^2\le C_n(s)\,{\|P\|}_X^2
\]
on the space $X=L^2[(0,+\infty);w]$,
with the generalized Laguerre weight function $w(t)=t^s{\E}^{-t}$ $(s>-1)$,
for polynomials $P\in W_n$, where
\begin{equation}\label{WnPOL}
W_n=\left\{P \biggm| P(t) = \sum_{\nu=0}^na_\nu t^\nu,\  
 a_\nu\ge 0,\ \nu=0,1,\ldots,n \right\}.
\end{equation}
 For such polynomials and $s\ge\left(\sqrt{5}-1\right)/2$, he proved that
$$\int_0^{+\infty} (P'(t))^2w(t)\,{\D}t
\le \frac{n^2}{(2n+s)(2n+s-1)}
\int_0^{+\infty}(P(t))^2w(t)\,{\D} t,$$
with equality for $P(t) = t^n$, and for $0\le s \le {1}/{2}$
\[
\int_0^{+\infty }(P'(t))^2w(t){\D} t
\le \frac{1}{(2+s)(1+s)}
\int_0^{+\infty} (P(t))^2w(t){\D} t.
\]

The ranges $-1<s < 0$ and $1/2 < s < s_1=(\sqrt{5}-1)/2$ are not covered by Varma's results.
 After attempts by several authors (cf.
\cite[pp.~417--419]{BC12}), this gap was filled by Milovanovi\'c \cite{RJ32}, who determined the best constant
$C_n(s)$ for all $s \in (-1,+\infty)$.

\begin{theorem}[\cite{RJ32}]\label{Milov85}
The best constant $C_n(s)$  is 
\[
C_n(s) =\left\{\begin{array}{ll}
\dfrac{1}{(2+s)(1+s)}&(-1<s\le s_n),\\[3mm]
\dfrac{n^2}{(2n+s)(2n+s-1)}&(s_n<s\le+\infty),	
\end{array}\right.
\]
where the sequence $\{s_n\}$ is defined by
\[
s_n=\frac{\sqrt{17 n^2+2 n+1}-3 n+1}{2(n+1)}.	
\]
\end{theorem}

As we can see, the best constant $s\mapsto C_n(s)$ is a continuous non-increasing function in $s$. The sequences $\{s_n\}$ is a decreasing, $s_1>s_2>s_3>\cdots\,$ and converges to
\[
\lim_{n\to+\infty}s_n=s_{\infty}=\frac{1}{2}(\sqrt{17}-3)\approx 0.56155281.	
\]
Otherwise, a few first terms of this  sequences are
\begin{alignat*}{3}
s_1&=\frac{1}{2} \left(\sqrt{5}-1\right)\approx 0.61803399,&\quad &s_2=\frac{1}{6} \left(\sqrt{73}-5\right)\approx0.59066729,	\\
s_3&=\frac{1}{2}\left(\sqrt{10}-2\right)\approx 0.58113883,	&\quad &s_4=\frac{1}{10}
   \left(\sqrt{281}-11\right)\approx0.57630546,\\
s_5&= \frac{1}{6}\left(\sqrt{109}-7\right)\approx0.57338442,
&\quad &s_6=\frac{4}{7}\approx0.57142857,\ \ \mbox{etc.}\end{alignat*}

It is clear that $C_n(s)\le C_{n+1}(s)$, because $W_{n}\subset W_{n+1}$, but for $-1<s<{1}/{2}$ the constants for different $n$ are the same. Notice, that for $n=1$ the same expression holds for each $s>-1$, i.e.,
 $C_1(s)=1/((2+s)(1+s))$. Also, it is interesting to mention that in the limit case when $n\to +\infty$, the best constant becomes the constant $1/4$ for each $s\ge s_{\infty}=(\sqrt{17}-3)/2\approx 0.56155281$.

The corresponding extremal problems for higher derivatives are also investigated, as well as for other weight functions
(for details see \cite{RJ50,RJ51} and  \cite[pp.~644--664] {mil1994}).
\medskip 

{\bf Acknowledgement.}
The work was supported in part by the Serbian Academy of Sciences and Arts ($\Phi$-96).

\end{document}